\documentclass[final,3p,times]{elsarticle}

\usepackage{amsmath,amssymb,amsthm}
\usepackage{bm,multirow}

\numberwithin{equation}{section}

\theoremstyle{plain}
\newtheorem{theorem}{Theorem}[section]
\newtheorem{proposition}[theorem]{Proposition}
\newtheorem{lemma}[theorem]{Lemma}
\newtheorem{corollary}[theorem]{Corollary}

\theoremstyle{remark}
\newtheorem{remark}[theorem]{Remark}
\newtheorem{assumption}[theorem]{Assumption}
\newtheorem{example}[theorem]{Example}

\theoremstyle{definition}

\newcommand{\gap}{\hspace{0.1cm}}

\newcommand{\cP}{\mathcal{P}}
\newcommand{\cS}{\mathcal{S}}
\newcommand{\cT}{\mathcal{T}}

\newcommand{\intO}{\int_{\Omega}}
\newcommand{\ES}{\mathrm{ES}}
\newcommand{\NS}{\mathrm{NS}}
\newcommand{\SSE}{\mathrm{SSE}}
\newcommand{\ASM}{\mathrm{ASM}}
\newcommand{\alt}{\mathrm{alt}}

\newcommand{\cN}{\mathcal{N}}

\let\div\relax
\DeclareMathOperator{\div}{div}

\journal{arXiv}

\begin{document}

\begin{frontmatter}
\title{Preconditioning for finite element methods with strain smoothing\tnoteref{acknowledgement}}
\tnotetext[acknowledgement]{
Chaemin Lee's work was supported by Basic Science Research Program through the National Research Foundation of Korea~(NRF) funded by the Ministry of Education~(No.~2021R1A6A3A01086822).
Jongho Park's work was supported by Basic Science Research Program through NRF funded by the Ministry of Education~(No.~2019R1A6A1A10073887).
}

\author[CL1,CL2]{Chaemin Lee\corref{cor}}
\address[CL1]{Department of Safety Engineering, Chungbuk National University, Chungbuk 28644, Korea}
\address[CL2]{Department of Mechanical Engineering, KAIST, Daejeon 34141, Korea}
\ead{clee@cbnu.ac.kr}

\author[JP]{Jongho Park}
\ead{jongho.park@kaist.ac.kr}
\ead[url]{https://sites.google.com/view/jonghopark}
\address[JP]{Natural Science Research Institute, KAIST, Daejeon 34141, Korea}

\cortext[cor]{Corresponding author}

\begin{abstract}
Strain smoothing methods such as the smoothed finite element methods~(S-FEMs) and the strain-smoothed element method~(SSE) have successfully improved the convergence behavior of finite elements.
The strain smoothing methods have been applied in numerous finite element analyses, especially for analyzing solids and structures; however, there have been no studies on efficient numerical solvers for the methods.
We need mathematically and numerically well-elaborated iterative solvers for efficient applications to large-scale problems.
In this study, we investigate how to design appropriate preconditioners for the methods with inspiration from the spectral properties of the strain smoothing methods.
First, we analyze the spectrums of the stiffness matrices of the edge-based S-FEM and SSE.
Subsequently, we propose improved two-level additive Schwarz preconditioners for the strain smoothing methods by modifying local solvers appropriately.
For convenience of implementation, an alternative form of the preconditioners is proposed by defining the coarse-scale operation in terms of the standard FEM. 
We verify our theoretical results through numerical experiments.
\end{abstract}


\end{frontmatter}


\section{Introduction}
\label{Sec:Introduction}
There have been various attempts to improve the performance of finite elements, among which strain smoothing methods can achieve the goal without introducing additional degrees of freedom.
Chen et al.~\cite{CWYY:2001} first proposed the concept of strain smoothing methods for the Galerkin mesh-free method.  
Subsequently, Liu et al.~\cite{LDN:2007} applied the strain smoothing technique to the finite element method~(FEM) and developed a series of smoothed FEMs~(S-FEMs). 
The S-FEMs are classified according to the construction of smoothing domains; the edge-based S-FEM~(ES-FEM) and node-based S-FEM~(NS-FEM) are well-known and broadly used~\cite{LNNL:2009,LNL:2009,LN:2010,HHB:2017,YWZJF:2019}.   
The ES-FEM generally exhibits the best convergence properties among the S-FEMs~\cite{LNL:2009}; the NS-FEM is effective in relieving volumetric locking~\cite{LNNL:2009}.
Several studies were conducted to establish the theoretical properties of the S-FEMs~\cite{LN:2010,NBN:2008,LNN:2010}.
Recently, the strain-smoothed element method~(SSE) has been developed~\cite{LL:2018,LL:2019,LKL:2021,LMP:2022,LLL:2022}. 
Whereas the S-FEMs construct strain fields for specifically defined smoothing domains, the SSE constructs strain fields for elements. 
The SSE provides a finite element solution with reduced discretization error by fully using the strains of all neighboring elements for strain smoothing.
A theoretical foundation for the convergence properties of the SSE has been established in~\cite{LP:2021}.

Although there has been a vast literature on the development of new strain smoothing methods and their applications to various engineering problems~(see~\cite{ZL:2018} for a recent survey), there have been no studies on efficient numerical solvers for the strain smoothing methods, to the best of our knowledge.
However, developing robust and efficient numerical solvers is critical for successful application of the methods to large-scale engineering problems~\cite{FR:1994}.
Particularly, iterative solvers are suitable for large-scale sparse linear problems~\cite{Saad:2003}.
Since the performance of iterative solvers relies on the condition number of a target linear system, an effective way to improve iterative solvers is to design good preconditioners.
In this perspective, there have been plenty of notable works on preconditioning of large-scale linear problems arising in structural mechanics; see, e.g.,~\cite{Smith:1992,KP:1998,GOS:2003,DW:2009}.

In this study, we examine how to design suitable preconditioners for the strain smoothing methods.
The main observation is that the stiffness matrices of the ES-FEM and SSE are spectrally equivalent to that of the standard FEM.
This observation guarantees that the ES-FEM and SSE can adopt any preconditioner designed for the standard FEM and enjoy the advantages of the preconditioner such as good conditioning or numerical scalability.
As a concrete example, we consider an overlapping Schwarz preconditioner, which is one of the most broadly used parallel preconditioners for finite element problems~\cite{TW:2005,DCPS:2013,Calvo:2019,CPS:2021}.
We prove that the standard two-level additive Schwarz preconditioner~\cite{TW:2005} designed for the standard FEM can be applied to the ES-FEM and SSE, satisfying the condition number bound $C( 1 + H/\delta)$, where $C$ is a positive constant independent of the mesh and subdomain sizes, $H$ is the subdomain size, and $\delta$ is the overlapping width for the overlapping domain decomposition associated with the additive Schwarz preconditioner.
Additionally, we propose novel improved two-level additive Schwarz preconditioners for the ES-FEM and SSE with better condition number estimates than the standard two-level additive Schwarz preconditioner.
With some simple modifications on the local problems of the standard Schwarz preconditioner, we obtain the proposed preconditioners that show improved performance in both theoretical and numerical senses.
The improvement strategy can be applied to not only additive Schwarz preconditioners but also a broad range of subspace correction preconditioners~\cite{Xu:1992,Xu:2001} such as multigrid and domain decomposition preconditioners.
Notably, several existing iterative solvers for linear systems fit into the framework of subspace correction~\cite{Xu:1992}; the improvement strategy introduced in this study reveals new possibilities for designing efficient iterative solvers for various contemporary FEMs.
Numerical results verify the theories presented in this study and prove the superiority of the proposed improved preconditioners.

This study includes an interesting remark on the NS-FEM; although the NS-FEM and ES-FEM are considered members of the common class of S-FEMs, their spectral properties may differ significantly.
In this study, we claim that the stiffness matrix of the NS-FEM may not be spectrally equivalent to that of the standard FEM.
Specifically, we present an example that the condition number $\kappa (K^{-1} \bar{K}_{\NS})$ increases as the mesh size $h$ decreases, where $K$ and $\bar{K}_{\NS}$ are the stiffness matrices of the standard FEM and NS-FEM, respectively.
This suggests the need to develop different mathematical theories for the NS-FEM and ES-FEM. 
However, most of the existing theories~\cite{LN:2010,LNN:2010,ZL:2018} are based on a unified S-FEM framework.

The remainder of this paper is organized as follows.
In Section~\ref{Sec:FEM}, we summarize the key features of the ES-FEM and SSE.
Section~\ref{Sec:Spectral} deals with the spectral properties of the ES-FEM and SSE; specifically, we demonstrate that the stiffness matrices of these methods are spectrally equivalent to that of the standard FEM.
Utilizing the spectral equivalence, in Section~\ref{Sec:Improvement}, we present efficient two-level additive Schwarz preconditioners for the ES-FEM and SSE and analyze their convergence properties.
Section~\ref{Sec:Numerical} presents numerical results that support the theoretical findings.
In Section~\ref{Sec:NS}, we give some remarks on the spectral property of the NS-FEM.
We conclude the study in Section~\ref{Sec:Conclusion}.

\section{Finite element methods with strain smoothing}
\label{Sec:FEM}
We provide brief descriptions of the S-FEM and SSE for a model Poisson problem
\begin{equation} \begin{split} \label{model}
- \Delta u = f \quad &\textrm{ in } \Omega, \\
u=0 \quad &\textrm{ on } \partial \Omega,
\end{split} \end{equation}
where $\Omega \subset \mathbb{R}^2$ is a bounded polygonal domain.
For simplicity, we consider the case of three-node triangular elements throughout this study; see~\cite{NLN:2011} and~\cite{LKL:2021} for formulations of polygonal finite elements adopting the S-FEM and SSE, respectively.
For a subregion $\omega \subset \Omega$ and a nonnegative integer $n$, the collection of all polynomials of degree less than or equal to $n$ on $\omega$ is denoted by $\cP_n (\omega)$.
Let $\cT_h$ be a quasi-uniform triangulation of $\Omega$ with a characteristic element diameter $h > 0$.
We define $V_h$ as the conforming piecewise linear finite element space on $\cT_h$, i.e.,
\begin{equation*}
V_h = \left\{ v \in H_0^1 (\Omega) :  v|_{e} \in \cP_1 (e) \gap \forall e \in \cT_h \right\},
\end{equation*}
where $H_0^1 (\Omega)$ is the usual Sobolev space consisting of all functions $u \in L^2 ( \Omega)$ such that $\nabla u \in (L^2 (\Omega))^2$ and $u|_{\partial \Omega} = 0$.
We also define $W_h$ as follows:
\begin{equation*}
W_h = \left\{ \epsilon \in (L^2 (\Omega))^2 : \epsilon|_{e} \in (\cP_0 (e))^2 \gap \forall e \in \cT_h \right\}.
\end{equation*}
Then we readily have that $v \in V_h$ implies $\nabla v \in W_h$.
With a slight abuse of notation, we do not distinguish between finite element functions and the corresponding vectors of degrees of freedom in the following.

\subsection{Standard finite element method}
The geometry of a 3-node triangular element $e \in \cT_h$ is interpolated by
\begin{equation*}
(x,y) = \sum_{i=1}^3 h_{i}(r,s) ( x_{i}, y_{i}) \in e,
\end{equation*}
where $(x_i, y_i)$, $1 \leq i \leq 3$, is the position vector of the $i$th node of $e$ in the global Cartesian coordinate system, and $h_{i}(r,s)$ is the two-dimensional interpolation function of the standard isoparametric procedure corresponding to the $i$th node, that is, $h_{1}(r,s)=1-r-s$, $h_{2}(r,s)=r$, and $h_{3}(r,s)=s$. 
The corresponding interpolation of the function $u$ within the element $e$ is given by
\begin{equation*}
u(x, y) = \sum_{i=1}^3 h_{i}(r,s) u_{i},
\end{equation*}
where $u_{i} = u(x_i, y_i)$.
Note that $u$ is continuous and piecewise linear on $\cT_h$, i.e., $u \in V_h$. 

The local gradient $\epsilon^{(e)}$ within element $e$ is obtained through the standard isoparametric finite element procedure as follows:
\begin{equation}
\label{local_gradient}
\epsilon^{(e)}  = B^{(e)} u^{(e)} \quad \textrm{with}  \quad
B^{(e)} = \begin{bmatrix} \frac{\partial h_{1}}{\partial x} & \frac{\partial h_{2}}{\partial x} & \frac{\partial h_{3}}{\partial x} \\
\frac{\partial h_{1}}{\partial y} & \frac{\partial h_{2}}{\partial y} & \frac{\partial h_{3}}{\partial y} \end{bmatrix}, \gap
u^{(e)} = \begin{bmatrix} u_{1} & u_{2} & u_{3} \end{bmatrix}^T.
 \end{equation}
 
The stiffness matrix $K$ corresponding to the standard FEM is given by
\begin{equation}
\label{stiffness_standard}
u^T K v = \intO \nabla u \cdot \nabla v \, d\Omega =  \sum_{e \in \cT_h}  |e| \left( B^{(e)} u^{(e)} \right)^T \left( B^{(e)} v^{(e)} \right), \quad u,v \in V_h.
\end{equation}
A finite element solution corresponding to the standard FEM is given by a solution of a linear system
\begin{equation*}
Ku = f,
\end{equation*}
where the load vector $f$ is defined as
\begin{equation}
\label{load_vector}
f^T v = \intO fv \,d \Omega, \quad v \in V_h.
\end{equation}

\begin{figure}[]
\centering
\includegraphics[width=1\textwidth]{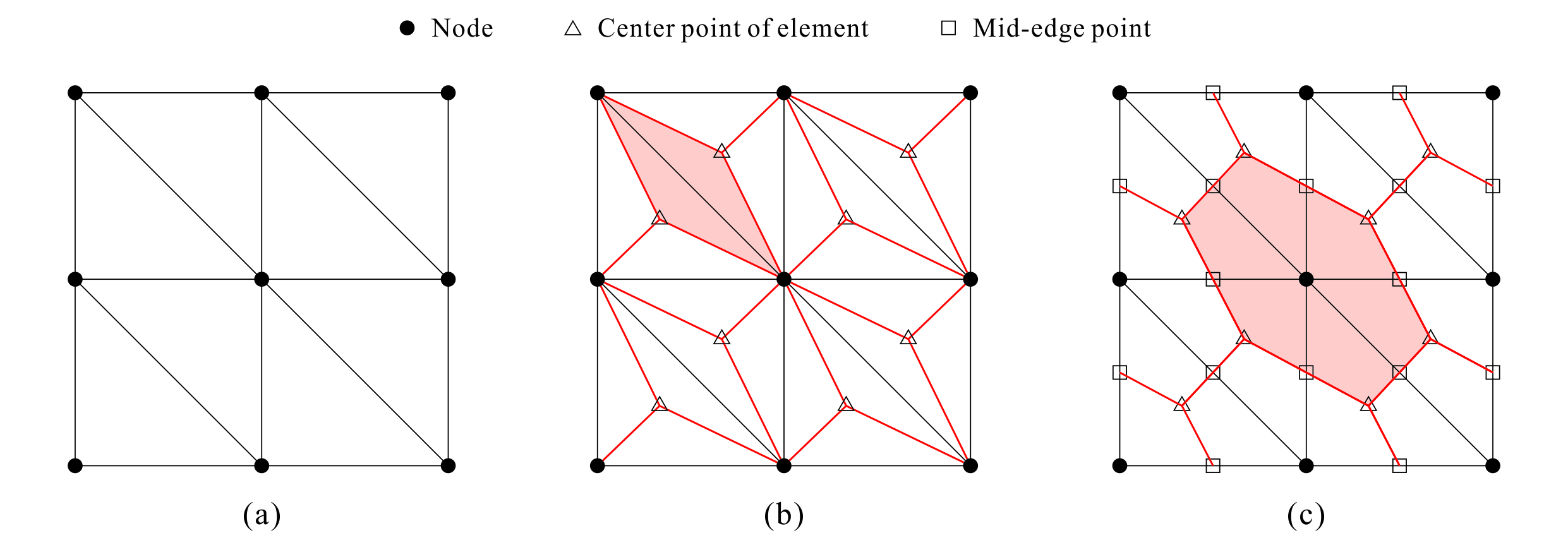}
\caption{Discretizations based on \textbf{(a)} finite elements, \textbf{(b)} edge-based smoothing domains, and \textbf{(c)} node-based smoothing domains.}
\label{Fig:domains}
\end{figure}

\subsection{Edge-based smoothed finite element method~(ES-FEM)}
\label{Subsec:ES}
The standard FEM discretizes a region into finite elements~(see~Fig.~\ref{Fig:domains}(a)), whereas S-FEM performs discretization based on newly defined smoothing domains.
The well-known S-FEMs are the ES-FEM and NS-FEM, which form the smoothing domains based on the edges and nodes of $\cT_h$, respectively. 
We briefly introduce the ES-FEM; see Section~\ref{Sec:NS} for a description of the NS-FEM.

\begin{figure}[]
\centering
\includegraphics[width=0.8\textwidth]{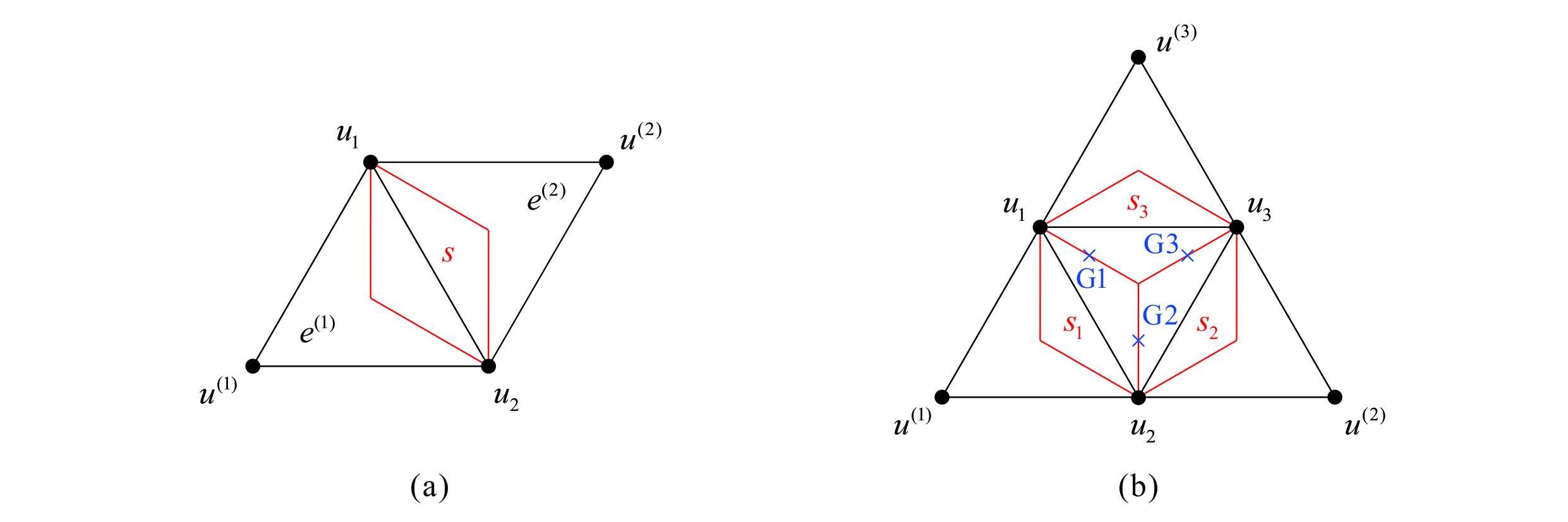}
\caption{Degrees of freedom of $u \in V_h$ corresponding to the vectors \textbf{(a)} $\bar{u}^{(s)}$ in~\eqref{local_gradient_ES} and \textbf{(b)} $\bar{u}^{(e)}$ in~\eqref{local_gradient_SSE}.}
\label{Fig:dofs}
\end{figure}

In the ES-FEM, each element in $\cT_h$ is divided into three triangular subdomains using its nodes and the barycenter~($r=s=1/3$).
Subsequently, the edge-based smoothing domains are defined as assemblages of two neighboring subdomains belonging to different elements; see~Fig.~\ref{Fig:domains}(b).
In the following, let $\cS_{h,\ES}$ denote the collection of all smoothing domains constructed from $\cT_h$ for the ES-FEM.
We define $W_{h, \ES}$ as the collection of all piecewise constant vector fields on $\cS_{h, \ES}$, i.e.,
\begin{equation*}
W_{h,\ES} = \left\{ \epsilon \in  (L^2 (\Omega))^2 : \epsilon|_e \in (\cP_0 (e))^2 \gap \forall e \in \cS_{h, \ES} \right\}.
\end{equation*}
The ES-FEM smoothing operator $S_{h,\ES} \colon W_h \rightarrow W_{h,\ES}$ that maps a given gradient field $\epsilon \in W_h$ to the corresponding smoothed gradient field $\bar{\epsilon} \in W_{h,\ES}$ is defined as follows. 
The local smoothed gradient $\bar\epsilon^{(s)}$ for a smoothing domain $s \in \cS_{h,\ES}$ is defined by
\begin{equation}
\label{ES_gradient}
\bar\epsilon^{(s)} = \frac{|e^{(1)}| \epsilon^{(e^{(1)})} + |e^{(2)}| \epsilon^{(e^{(2)})}}{ |e^{(1)}| + |e^{(2)}|},
\end{equation}
where $e^{(1)}$ and $e^{(2)}$ are the elements in $\cT_h$ sharing the edge corresponding to $s$, and $\epsilon^{(e^{(1)})}$ and $\epsilon^{(e^{(2)})}$ were defined in~\eqref{local_gradient}.
The local smoothed gradient in~\eqref{ES_gradient} can be expressed in a matrix-vector form as
\begin{equation}
\label{local_gradient_ES}
\bar\epsilon^{(s)} = \bar{B}_{\ES}^{(s)} \bar{u}^{(s)}
\quad \textrm{with} \quad
 \bar{B}_{\ES}^{(s)} =  \frac{|e^{(1)}|}{|e^{(1)}| + |e^{(2)}|} B^{(e_1)} R^{(e_1)}
+ \frac{|e^{(2)}|}{|e^{(1)}| + |e^{(2)}|} B^{(e_2)} R^{(e_2)} , \gap
\bar{u}^{(s)} = \begin{bmatrix}u_1 & u_2 & u^{(1)} & u^{(2)} \end{bmatrix}^T,
\end{equation}
where the vector $\bar{u}^{(s)}$ consists of the four degrees of freedom of $u \in V_h$ at the nodes of the elements sharing the edge corresponding to $s$ as shown in Fig.~\ref{Fig:dofs}(a), $B^{(e_1)}$ and $B^{(e_2)}$ were defined in~\eqref{local_gradient}, and $R^{(e_1)} $ and $R^{(e_2)}$ are boolean matrices that extract the degrees of freedom corresponding to the elements $e_1$ and $e_2$, respectively, i.e.,
\begin{equation*}
R^{(e_1)} = \begin{bmatrix} 1 & 0 & 0 & 0 \\ 0 & 1 & 0 & 0 \\ 0 & 0 & 1 & 0 \end{bmatrix}, \gap
R^{(e_2)} = \begin{bmatrix} 1 & 0 & 0 & 0 \\ 0 & 1 & 0 & 0 \\ 0 & 0 &  0 & 1 \end{bmatrix}.
\end{equation*}

Finally, the stiffness matrix $\bar{K}_{\ES}$ for the ES-FEM can be obtained as
\begin{equation}
\label{stiffness_ES}
u^T \bar{K}_{\ES} v = \intO \bar{\nabla}_{\ES} u \cdot \bar{\nabla}_{\ES} v \,d\Omega
= \sum_{s \in \cS_{h,\ES}} |s| \left( \bar{B}_{\ES}^{(s)} \bar{u}^{(s)} \right)^T \left( \bar{B}_{\ES}^{(s)} \bar{v}^{(s)} \right),
\quad u,v \in V_h,
\end{equation}
where $\bar{\nabla}_{\ES}$ denotes the global smoothed gradient operator corresponding to~\eqref{ES_gradient}, i.e., $\bar{\nabla}_{\ES} = S_{h, \ES} \nabla$.
We obtain an ES-FEM finite element solution by solving a linear system
\begin{equation*}
\bar{K}_{\ES} u = f,
\end{equation*}
where the load vector $f$ was given in~\eqref{load_vector}.
That is, the ES-FEM uses an alternative stiffness matrix $\bar{K}_{\ES}$, whereas its load vector $f$ is the same as that of the standard FEM.
Among the various types of S-FEMs, it has been numerically verified that the ES-FEM is the most effective method in reducing the discretization error of the finite element solution; see, e.g.,~\cite{LNL:2009}.

\subsection{Strain-smoothed element method~(SSE)}
\label{Subsec:SSE}
When the SSE is employed, a smoothed gradient field is constructed for each element in $\cT_h$ and the gradient information in all elements adjacent to a target element is utilized.
No smoothing domains are required; the domain discretization with the SSE is the same as the one with the standard FEM.
The space $W_{h,\SSE}$ for the SSE smoothed gradient fields is given by
\begin{equation*}
W_{h,\SSE} = \left\{ \epsilon \in (L^2(\Omega))^2 : \epsilon|_e \in (\cP_1 (e))^2 \gap \forall e \in \cT_h \right\}.
\end{equation*}
Note that $W_{h,\SSE}$ includes piecewise linear polynomials, while $W_h$ only includes piecewise constant functions.
We present how to construct the SSE smoothed gradient $\bar{\epsilon} = S_{h,\SSE} \epsilon \in W_{h,\SSE}$ for $\epsilon \in W_h$, where $S_{h,\SSE} \colon W_h \rightarrow W_{h,\SSE}$ denotes the SSE smoothing operator.
For an element $e \in \cT_h$, there can be up to three neighboring elements in $\cT_h$ through element edges, say $e^{(k)}$, $1\leq k \leq 3$.
An intermediate smoothed gradient $\hat\epsilon^{(k)}$ between $e$ and its neighboring element $e^{(k)}$ is defined by
\begin{equation}
\label{intermediate_strain}
\hat\epsilon^{(k)} = \frac{|e| \epsilon^{(e)}+ |e^{(k)}| \epsilon^{(e^{(k)})}}{|e| + |e^{(k)}|},
\end{equation}
where  $\epsilon^{(e)}$ and $\epsilon^{(e^{(k)})}$ were defined in~\eqref{local_gradient}.
If there is no adjacent element $e^{(k)}$ for some $k$, we simply use $\hat\epsilon^{(k)}=\epsilon^{(e)}$.
Subsequently, we construct a linear smoothed gradient field $\bar{\epsilon}^{(e)}$ on the target element $e$ by unifying the intermediate smoothed gradients in~\eqref{intermediate_strain}.
The values are assigned at three Gaussian integration points $\mathrm{G}k$, $1 \leq k \leq 3$, of $e$ as the pointwise values of $\bar{\epsilon}^{(e)}$ as follows:
\begin{equation}
\label{SSE_gradient}
\bar{\epsilon}^{(e)}(\mathrm{G}k) = \frac{\hat\epsilon^{(k-1)}+\hat\epsilon^{(k)}}{2},
\end{equation}
with the convention $\hat\epsilon^{(0)}=\hat\epsilon^{(3)}$.
The smoothed gradient field $\bar{\epsilon}^{(e)}$ is uniquely determined within $e$ by linear interpolation of the pointwise values.
The local smoothed gradient in~\eqref{SSE_gradient} can be expressed in a matrix-vector form as
\begin{equation}
\label{local_gradient_SSE}
\bar\epsilon^{(e)} = \bar{B}_{\SSE}^{(e)} \bar{u}^{(e)}
\quad \textrm{with} \quad
\bar{B}_{\SSE}^{(e)} = \frac{1}{2} \begin{bmatrix}
\bar{B}_{\ES}^{(s_1)} R^{(s_1)} + \bar{B}_{\ES}^{(s_3)} R^{(s_3)} \\
\bar{B}_{\ES}^{(s_1)} R^{(s_1)} + \bar{B}_{\ES}^{(s_2)} R^{(s_2)} \\
\bar{B}_{\ES}^{(s_2)} R^{(s_2)} + \bar{B}_{\ES}^{(s_3)} R^{(s_3)}
\end{bmatrix}, \gap
\bar{u}^{(e)} = \begin{bmatrix} u_1 & u_2 & u_3 & u^{(1)} & u^{(2)} & u^{(3)} \end{bmatrix}^T, 
\end{equation}
where the vector $\bar{\epsilon}^{(e)}$ comprises three pointwise values $\bar{\epsilon}^{(e)} (\mathrm{G}1)$, $\bar{\epsilon}^{(e)} (\mathrm{G}2)$, and $\bar{\epsilon}^{(e)} (\mathrm{G}3)$, the vector $\bar{u}^{(e)}$ consists of at most six degrees of freedom of $u \in V_h$ at the nodes of $e$ and its neighboring elements~(see Fig.~\ref{Fig:dofs}(b)), the matrices $\bar{B}_{\ES}^{(s_k)}$, $1 \leq k \leq 3$, were defined in~\eqref{local_gradient_ES}, and $R^{(s_k)}$ are boolean matrices that extract the degrees of freedom corresponding to the subdomains $s_k$, i.e.,
\begin{equation*}
R^{(s_1)} = \begin{bmatrix} 1 & 0 & 0 & 0 & 0 & 0 \\ 0 & 1 & 0 & 0 & 0 & 0 \\ 0 & 0 & 0 & 1 & 0 & 0 \\ 0 & 0 & 1 & 0 & 0 & 0 \end{bmatrix}, \gap
R^{(s_2)} = \begin{bmatrix} 0 &1 & 0 & 0 & 0 & 0 \\ 0 & 0 & 1 & 0 & 0 & 0 \\ 0 & 0 & 0 & 0 & 1 & 0 \\ 1 & 0 & 0 & 0 & 0 & 0 \end{bmatrix}, \gap
R^{(s_3)} = \begin{bmatrix} 0 & 0 & 1 & 0 & 0 & 0 \\ 1 & 0 & 0 & 0 & 0 & 0 \\ 0 & 0 & 0 & 0 & 0 & 1 \\ 0 & 1 & 0 & 0 & 0 & 0 \end{bmatrix}.
\end{equation*}

In the SSE, the smoothed gradient field is constructed for each element in $\cT_h$.
The stiffness matrix $\bar{K}_{\SSE}$ for the SSE is calculated by
\begin{equation}
\label{stiffness_SSE}
u^T \bar{K}_{\SSE} v = \intO \bar{\nabla}_{\SSE} u \cdot \bar{\nabla}_{\SSE} v \,d\Omega
= \sum_{e \in \cT_h} \frac{|e|}{3} \left( \bar{B}_{\SSE}^{(e)} \bar{u}^{(e)} \right)^T \left( \bar{B}_{\SSE}^{(e)} \bar{v}^{(e)} \right),
\quad u,v \in V_h,
\end{equation}
where $\bar{\nabla}_{\SSE}$ denotes the global smoothed gradient operator corresponding to~\eqref{SSE_gradient}, i.e., $\bar{\nabla}_{\SSE} = S_{h,\SSE} \nabla$.
An SSE finite element solution is given by a solution of a linear system
\begin{equation*}
\bar{K}_{\SSE} u = f,
\end{equation*}
where the load vector $f$ was defined in~\eqref{load_vector}.

The strain-smoothed elements adopting the SSE have been verified to pass the three basic tests~(zero energy mode, isotropic element, and patch tests) and show improved convergence behaviors compared with other competitive elements in various numerical problems~\cite{LL:2018,LL:2019,LKL:2021}.

\begin{remark}
\label{Rem:consistent}
The ES-FEM and SSE are variationally consistent numerical schemes in the sense that they are derived from a particular variational principle for~\eqref{model}.
More precisely, it was proven in~\cite{LP:2021} that the ES-FEM and SSE are conforming Galerkin approximations of the following mixed variational principle: find $(u, \epsilon_1, \epsilon_2, \sigma_1, \sigma_2) \in V \times W \times W \times W \times W$ such that
\begin{equation} \begin{split} \label{VP}
\intO \sigma_1 \cdot \nabla v \,d \Omega + \intO ( - \sigma_1 + \sigma_2) \cdot \delta_1 \,d \Omega + \intO (\epsilon_2 - \sigma_2 ) \cdot \delta_2 \,d \Omega = \intO fv \,d \Omega
\quad &\forall v \in V, \gap \delta_1 \in W, \gap \delta_2 \in W, \\
\intO \tau_1 \cdot (\nabla u - \epsilon_1) \,d \Omega + \intO \tau_2 \cdot (\epsilon_1 - \epsilon_2) \,d \Omega = 0 
\quad &\forall \tau_1 \in W, \gap \tau_2 \in W,
\end{split} \end{equation}
where $V = H_0^1 (\Omega)$ and $W = (L^2 (\Omega))^2$; the equivalence between~\eqref{model} and~\eqref{VP} is presented in~\cite[Proposition~4.1]{LP:2021}.
Because of variational consistency, the existence and uniqueness of a solution for both methods can be proven and a discretization error bound can be obtained by invoking the standard theory of mixed FEMs~\cite{Ciarlet:2002,BS:2008}; see~\cite{LP:2021} for details.
\end{remark}

\section{Spectral equivalence among stiffness matrices}
\label{Sec:Spectral}
In this section, we prove that the stiffness matrices of the ES-FEM and SSE defined in~\eqref{stiffness_ES} and~\eqref{stiffness_SSE}, respectively, are spectrally equivalent to that of the standard FEM defined in~\eqref{stiffness_standard}.
The results of this section imply that the ES-FEM and SSE can adopt any preconditioner designed for the standard FEM without degrading the performance of the preconditioner.
In this sense, they are advantageous for use with preconditioned iterative schemes such as the preconditioned conjugate gradient method and other Krylov space methods~\cite{Saad:2003} compared to other FEMs with strain smoothing.
In contrast, in Section~\ref{Sec:NS}, we will demonstrate that the stiffness matrix of the NS-FEM is not spectrally equivalent to that of the standard FEM in general.

First, we present a simple but useful lemma that is required for the spectral analysis of the ES-FEM and SSE.

\begin{lemma}
\label{Lem:zero}
Let $\omega_1$ and $\omega_2$ be polygonal regions in $\mathbb{R}^2$ sharing an edge $f$, i.e., $\overline{\omega}_1 \cap \overline{\omega}_2 = f$.
If a continuous and piecewise linear function $u \colon \overline{\omega_1 \cup \omega_2} \rightarrow \mathbb{R}$ satisfies $w_1 \nabla u|_{\omega_1} + w_2 \nabla u|_{\omega_2} = 0$ for some $w_1$, $w_2 > 0$, then it is constant along $f$.
\end{lemma}
\begin{proof}
Let $t$ be a unit vector along the direction of the edge $f$.
If we suppose that $u$ is not constant along $f$, then it follows that $\nabla u|_{\omega_1} \cdot t = \nabla u|_{\omega_2} \cdot t \neq 0$.
However, it contradicts $(w_1 \nabla u|_{\omega_1} + w_2 \nabla u|_{\omega_2}) \cdot t = 0$.
\end{proof}

Using Lemma~\ref{Lem:zero} and the fact that the strain smoothing operation of the ES-FEM is an orthogonal projection in $(L^2 (\Omega))^2$~\cite{LNN:2010}, we can prove the spectral equivalence between the stiffness matrices of the standard FEM and ES-FEM as follows.

\begin{theorem}
\label{Thm:ES}
The stiffness matrices $K$ and $\bar{K}_{\ES}$ of the standard FEM and ES-FEM defined in~\eqref{stiffness_standard} and~\eqref{stiffness_ES}, respectively, are spectrally equivalent.
That is, there exists two positive constants $\underline{C}$ and $\overline{C}$ independent of the mesh size $h$ such that
\begin{equation*}
\underline{C} u^T K u \leq u^T \bar{K}_{\ES} u \leq \overline{C} u^T K u \quad  \forall u \in V_h. 
\end{equation*}
\end{theorem}
\begin{proof}
We recall that the ES-FEM smoothing operator $S_{h,\ES}$ defined in Section~\ref{Subsec:ES} is an $(L^2 (\Omega))^2$-orthogonal projection.
More precisely, it was shown in~\cite[Remark~4]{LNN:2010} that
\begin{equation*}
S_{h,\ES} \epsilon = P_{h,\ES} \epsilon \quad \forall \epsilon \in W_h,
\end{equation*}
where $P_{h,\ES} \colon (L^2 (\Omega))^2 \rightarrow W_{h,\ES}$ is the $(L^2 (\Omega))^2$-orthogonal projection onto $W_{h,\ES}$.
Then it follows that
\begin{equation*}
u^T \bar{K}_{\ES} u = \intO |\bar{\nabla}_{\ES} u|^2 \,d\Omega
= \intO |P_{h,\ES} \nabla u|^2 \,d \Omega
 \leq \intO |\nabla u|^2 \,d\Omega = u^T K u \quad \forall u \in V_h.
\end{equation*}
Consequently, we have $\overline{C} = 1$.
Next, we estimate $\underline{C}$.
For any element $e \in \cT_h$ intersecting three smoothing domains $s^{(k)} \in \cS_{h,\ES}$, $1 \leq k \leq 3$, we have
\begin{equation*}
\int_e |\bar{\nabla}_{\ES} u|^2 \,d\Omega
= \sum_{k=1}^3 \int_{e \cap s^{(k)}} |\bar{\nabla}_{\ES} u|^2 \,d\Omega
= \frac{|e|}{3} \sum_{k=1}^3 (\hat{\epsilon}^{(k)})^T \hat{\epsilon}^{(k)},
\end{equation*}
where $\hat{\epsilon}^{(k)}$ was given in~\eqref{intermediate_strain}.
In order for $\int_e |\bar{\nabla}_{\ES} u|^2 \,d\Omega$ to be zero, we must have $\hat{\epsilon}^{(k)} = 0$ for all $k$.
If $s^{(k)} \subset e$ for some $k$, i.e., if $e$ is a boundary element, then we have
\begin{equation*}
\epsilon^{(e)} = \hat{\epsilon}^{(k)} = 0,
\end{equation*}
where $\epsilon^{(e)}$ was defined in~\eqref{local_gradient}.
If $e$ is an interior element with three adjacent elements, Lemma~\ref{Lem:zero} implies that $u$ is constant along all the edges of $e$, such that $u$ is constant on $e$ and
\begin{equation*}
\int_e |\nabla u|^2 \,d\Omega = |e|(\epsilon^{(e)})^T \epsilon^{(e)} = 0, 
\end{equation*}
Up to this point, we have shown that
\begin{equation*}
\int_e |\bar{\nabla}_{\ES} u|^2 \,d\Omega = 0  \quad \textrm{implies} \quad \int_e |\nabla u|^2 \,d\Omega = 0.
\end{equation*}
Hence there exists a positive constant $C_e$ such that
\begin{equation}
\label{ES_local}
\int_e |\bar{\nabla}_{\ES} u|^2 \,d\Omega
\geq C_e \int_e |\nabla u|^2 \,d\Omega  \quad \forall u \in V_h.
\end{equation}

Now, we verify that the constant $C_e$ in~\eqref{ES_local} is independent of the mesh size $h$ using a scaling argument~(cf.~\cite[Section~3.4]{TW:2005}).
The transformation $x = h\hat{x}$ maps the domain $\hat{\Omega} = h^{-1}\Omega$ with the same shape as $\Omega$ into $\Omega$.
The domain $\hat{\Omega}$ naturally admits a triangulation
\begin{equation*}
\hat{\cT}_1 = \left\{ \hat{e} = h^{-1}e : e \in \cT_h \right\},
\end{equation*}
whose characteristic element diameter is $1$.
Spaces $\hat{V}_1$ and $\hat{W}_1$ are defined as
\begin{align*}
\hat{V}_1 &= \left\{ \hat{v} \in H_0^1 (\hat{\Omega}) : \hat{v}|_{\hat{e}} \in \cP_1 (\hat{e}) \gap \forall \hat{e} \in \hat{\cT}_1 \right\}, \\
\hat{W}_1 &= \left\{ \hat{\epsilon} \in (L^2 (\hat{\Omega}))^2 : \hat{\epsilon}|_{\hat{e}} \in (\cP_0 (\hat{e}))^2 \gap \forall \hat{e} \in \hat{\cT}_1 \right\}.
\end{align*}
We define an ES-FEM smoothing operator $\hat{S}_{1,\ES}$ on $\hat{W}_1$ in the same manner as $S_{h,\ES}$.
The inequality~\eqref{ES_local} is valid for $\hat{e} = h^{-1} e \in \hat{\cT}_1$, with a constant $\hat{C}_{\hat{e}}$ that only depends on $\hat{e}$ and $\hat{e}^{(k)} = H^{-1} e^{(k)}$~($1 \leq k \leq 3$), i.e., only on the geometries of $e$ and $e^{(k)}$:
\begin{equation}
\label{ES_local_scaled}
\int_{\hat{e}} |\hat{S}_{1,\ES} \nabla \hat{u}|^2 \,d \hat{\Omega} \geq \hat{C}_{\hat{e}} \int_{\hat{e}} |\nabla \hat{u}|^2 \, d\hat{\Omega} \quad \forall \hat{u} \in \hat{V}_1.
\end{equation}
We observe that
\begin{equation}
\label{scaling_argument}
\bar{\nabla}_{\ES}u |_{e \cap e^{(k)}} = \frac{|e|}{|e| + |e^{(k)}|} \nabla u|_e + \frac{|e^{(k)}|}{|e| + |e^{(k)}|} \nabla u|_{e^{(k)}}
= \frac{|\hat{e}|}{|\hat{e}| + |\hat{e}^{(k)}|} h^{-1} \nabla \hat{u}|_{\hat{e}} + \frac{|\hat{e}^{(k)}|}{|\hat{e}| + |\hat{e}^{(k)}|} h^{-1} \nabla \hat{u}|_{\hat{e}^{(k)}} = h^{-1} \hat{S}_{1,\ES} \nabla \hat{u} |_{\hat{e} \cap \hat{e}^{(k)}},
\end{equation}
where $\hat{u}(\hat{x}) = u(h\hat{x})$ is a transformed function.
It follows that
\begin{equation*}
\int_e |\bar{\nabla}_{\ES} u|^2 \,d\Omega
\stackrel{\eqref{scaling_argument}}{=} \int_{\hat{e}} |h^{-1} \hat{S}_{1,\ES} \nabla \hat{u}|^2 h^2 \, d \hat{\Omega}
\stackrel{\eqref{ES_local_scaled}}{\geq} \hat{C}_{\hat{e}} \int_{\hat{e}} |\nabla \hat{u}|^2 \,d \hat{\Omega}
= \hat{C}_{\hat{e}} \int_{e} |h \nabla u|^2 h^{-2} \,d \Omega.
\end{equation*}
Hence, the inequality~\eqref{ES_local} holds with $C_e = \hat{C}_{\hat{e}}$, i.e., $C_e$ is independent of the mesh size $h$ and depends only on the geometries of $e$ and $e^{(k)}$.

Since $\cT_h$ is quasi-uniform, the constant $C_e$ in~\eqref{ES_local} has a uniform positive lower bound, say $\underline{C}$, over all $e \in \cT_h$.
For any $u \in V_h$, we have
\begin{equation*}
u^T \bar{K}_{\ES} u 
= \sum_{e \in \cT_h} \int_e |\bar{\nabla}_{\ES} u|^2 \,d\Omega
\geq \underline{C} \sum_{e \in \cT_h} \int_e |\nabla u|^2 \,d\Omega
= \underline{C} u^T K u,
\end{equation*}
which completes the proof.
\end{proof}

A useful consequence of Theorem~\ref{Thm:ES} is that any preconditioner for the standard FEM works for the ES-FEM.
Corollary~\ref{Cor:ES} presents a rigorous statement on the performance of a preconditioner applied to the ES-FEM.
Note that, for a symmetric and positive definite matrix $A$, $\kappa (A)$ denotes the condition number of $A$, i.e.,
\begin{equation*}
\kappa (A) = \frac{\lambda_{\max}(A)}{\lambda_{\min}(A)},
\end{equation*}
where $\lambda_{\min} (A)$ and $\lambda_{\max} (A)$ are the minimum and maximum eigenvalues of $A$, respectively.

\begin{corollary}
\label{Cor:ES}
Any preconditioner $M^{-1}$ for the standard FEM $K u = f$ works for the ES-FEM $\bar{K}_{\ES} u = f$ as well.
More precisely, there exists a positive constant $C$ independent of the mesh size $h$ such that
\begin{equation*}
\kappa (M^{-1} \bar{K}_{\ES}) \leq C \kappa (M^{-1} K).
\end{equation*}
\end{corollary}
\begin{proof}
By~\cite[Corollary~C.2]{TW:2005}, it follows that
\begin{equation*}
\kappa (M^{-1} \bar{K}_{\ES}) \leq \kappa (K^{-1} \bar{K}_{\ES})  \kappa (M^{-1} K) \leq \frac{\overline{C}}{\underline{C}}  \kappa (M^{-1} K),
\end{equation*}
where $\underline{C}$ and $\overline{C}$ were given in Theorem~\ref{Thm:ES}.
Setting $C = \overline{C}/\underline{C}$ completes the proof.
\end{proof}

We provide a detailed explanation for Corollary~\ref{Cor:ES}.
Suppose that we have a preconditioner $M^{-1}$ for the standard FEM such that $\kappa ( M^{-1} K) = O(h^{-\alpha})$ for some $\alpha \geq 0$.
Then, Corollary~\ref{Cor:ES} implies that preconditioning the stiffness matrix of the ES-FEM  by $M^{-1}$ yields the same condition number estimate as the standard FEM, i.e., $\kappa ( M^{-1} \bar{K}_{\ES}) = O(h^{-\alpha})$.
Therefore, it is ensured that any preconditioned iterative scheme for the ES-FEM reaches a target accuracy within the same number of iterations up to a multiplicative constant as the case of the standard FEM.

Similar results can be obtained for the SSE.
The spectral equivalence between the stiffness matrices of the standard FEM and SSE can be deduced by invoking the fact that the strain smoothing step of the SSE can be represented as a composition of orthogonal projection operators among some assumed strain spaces~\cite{LP:2021}.

\begin{theorem}
\label{Thm:SSE}
The stiffness matrices $K$ and $\bar{K}_{\SSE}$ of the standard FEM and SSE defined in~\eqref{stiffness_standard} and~\eqref{stiffness_SSE}, respectively, are spectrally equivalent.
That is, there exists two positive constants $\underline{C}$ and $\overline{C}$ independent of the mesh size $h$ such that
\begin{equation*}
\underline{C} u^T K u \leq u^T \bar{K}_{\SSE} u \leq \overline{C} u^T K u, \quad  u \in V_h. 
\end{equation*}
\end{theorem}
\begin{proof}
Let $\cS_{h, \SSE}$ be the collection of quadrilaterals formed by joining the centroid and midpoints of the edges of each element in $\cT_h$; see~\cite[Figure~3(c)]{LP:2021}.
The collection of all piecewise constant vector fields on $\cS_{h,\SSE}$ is denoted by $W_{h,\SSE}'$, i.e.,
\begin{equation*}
W_{h,\SSE}' = \left\{ \epsilon \in (L^2 (\Omega))^2 : \epsilon|_e \in (\cP_0 (e))^2 \gap \forall e \in \cS_{h, \SSE} \right\}.
\end{equation*}
Additionally, let $P_{h,\SSE} \colon (L^2 (\Omega))^2 \rightarrow W_{h, \SSE}'$ denote the $(L^2 (\Omega))^2$-orthogonal projection onto $W_{h, \SSE}'$.
Then the SSE smoothing operator $S_{h,\SSE}$ defined in Section~\ref{Subsec:SSE} satisfies the following equality~\cite[Theorem~3.3]{LP:2021}:
\begin{equation*}
\intO |S_{h,\SSE} \epsilon|^2 \,d \Omega = \intO |P_{h,\SSE} P_{h,\ES} \epsilon|^2 \,d \Omega \quad \forall \epsilon \in W_h,
\end{equation*}
where $P_{h,\ES}$ was defined in the proof of Theorem~\ref{Thm:SSE}.
Hence, we deduce that $u^T \bar{K}_{\SSE} u \leq u^T K u$ for all $u \in V_h$, i.e., $\overline{C} =1$.
In order to prove the $\underline{C}$-inequality, it suffices to show that 
\begin{equation}
\label{SSE_local}
\int_e |\bar{\nabla}_{\SSE} u|^2 \,d\Omega = 0  \quad \textrm{implies} \quad \int_e |\nabla u|^2 \,d\Omega = 0
\end{equation}
for any $e \in \cT_h$; if we show~\eqref{SSE_local}, we can deduce the $\underline{C}$-inequality by the same argument as in Theorem~\ref{Thm:ES}.
We take an element $e \in \cT_h$ and suppose that $\int_e |\bar{\nabla}_{\SSE} u|^2 \,d\Omega = 0$.
Observing that
\begin{equation*}
\int_e |\bar{\nabla}_{\SSE} u|^2 \,d\Omega = \frac{|e|}{3} \sum_{k=1}^3 \left( \frac{\hat\epsilon^{(k-1)}+\hat\epsilon^{(k)}}{2} \right)^T \left( \frac{\hat\epsilon^{(k-1)}+\hat\epsilon^{(k)}}{2} \right),
\end{equation*}
where $\hat{\epsilon}^{(k)}$ was defined in~\eqref{intermediate_strain}, we have $\hat{\epsilon}^{(k-1)} + \hat{\epsilon}^{(k)} = 0$ for all $k$.
Equivalently, we get $\hat{\epsilon}^{(k)} = 0$ for all $k$.
If $e$ is a boundary element, i.e., $\epsilon^{(e)} = \hat{\epsilon}^{(k)}$ for some $k$, we readily obtain $\epsilon^{(e)} = \hat{\epsilon}^{(k)} = 0$.
If $e$ is an interior element, invoking Lemma~\ref{Lem:zero}, we deduce that $u$ is constant on $e$ such that $\int_e |\nabla u|^2 \,d\Omega = 0$.
Therefore,~\eqref{SSE_local} holds.
\end{proof}

The following corollary is a direct consequence of Theorem~\ref{Thm:SSE}; it says that any preconditioner for the standard FEM is also well-suited for the SSE.
Corollary~\ref{Cor:SSE} is derived in the same manner as Corollary~\ref{Cor:ES}.

\begin{corollary}
\label{Cor:SSE}
Any preconditioner $M^{-1}$ for the standard FEM $K u = f$ works for the SSE $\bar{K}_{\SSE} u = f$ as well.
More precisely, there exists a positive constant $C$ independent of the mesh size $h$ such that
\begin{equation*}
\kappa (M^{-1} \bar{K}_{\SSE}) \leq C \kappa (M^{-1} K).
\end{equation*}
\end{corollary}

We present another useful consequence of Theorems~\ref{Thm:ES} and~\ref{Thm:SSE}: a Poincar\'{e}--Friedrichs-type inequality for the ES-FEM and SSE.
Poincar\'{e}--Friedrichs-type inequalities are especially useful for convergence analysis of FEMs with strain smoothing; see, e.g.,~\cite{LNN:2010,LP:2021}.

\begin{proposition}
\label{Prop:Poincare}
Let $\bar{\nabla}$ denote the smoothed gradient operator for either the ES-FEM or SSE. Then, there exists a positive constant $C$ independent of the mesh size $h$ such that
\begin{equation*}
\intO | \bar{\nabla} u |^2 \,d \Omega \geq C \intO u^2 \,d\Omega, \quad u \in V_h.
\end{equation*}
\end{proposition}
\begin{proof}
Combining the standard Poincar\'{e}--Friedrichs inequality~\cite[Lemma~A.14]{TW:2005} with Theorems~\ref{Thm:ES} and~\ref{Thm:SSE} yields the desired result.
\end{proof}

\begin{remark}
\label{Rem:coercivity}
Proposition~\ref{Prop:Poincare} indicates that the bilinear form
\begin{equation*}
\bar{a}(u,v) = \intO \bar{\nabla}u \cdot \bar{\nabla} v \,d\Omega, \quad u,v \in V_h
\end{equation*}
is coercive.
Coercivity of the bilinear form $\bar{a}(\cdot, \cdot)$ corresponding to the S-FEMs was first proven in~\cite{Liu:2010} using a positivity relay argument.
We note that the coercivity constant $C$ in Proposition~\ref{Prop:Poincare} is proven to be independent of $h$, whereas that in~\cite{Liu:2010} was not.
Thus, Proposition~\ref{Prop:Poincare} provides a sharper result than~\cite{Liu:2010}.
\end{remark}

\section{Improvement of preconditioners}
\label{Sec:Improvement}
As we observed in Section~\ref{Sec:Spectral}, existing preconditioners for the standard FEM can be applied to the ES-FEM and SSE, inheriting good convergence properties from the case of the standard FEM.
Meanwhile, when applied to the ES-FEM and SSE, the performance of preconditioners based on subspace correction~\cite{Xu:1992,Xu:2001} can be further improved by modifying local solvers appropriately.
In this section, we present how to construct improved subspace correction preconditioners for  the ES-FEM and SSE.
Specifically, we propose two-level additive Schwarz preconditioners~\cite{TW:2005} for the ES-FEM and SSE; note that Schwarz preconditioning is a standard methodology of parallel computing for large-scale finite element problems; see, e.g.,~\cite{DCPS:2013,Calvo:2019,CPS:2021}.
Although we utilize the two-level additive Schwarz preconditioner as descriptive examples, the method of improvement introduced in this section can be applied to various subspace correction preconditioners such as multigrid and domain decomposition preconditioners.
Throughout this section, we omit the subscript $h$ standing for the mesh size if there is no ambiguity.

\subsection{Two-level additive Schwarz preconditioner}
First, we summarize key features of the standard two-level additive Schwarz preconditioner for the standard FEM~\cite{TW:2005}.
Assuming that the domain $\Omega$ admits a coarse triangulation $\cT_H$ with the characteristic element diameter $H$, it is decomposed into $\cN$ nonoverlapping subdomains $\{\Omega_j \}_{j=1}^\cN$ such that each $\Omega_j$ is the union of several coarse elements in $\cT_H$, and the number of coarse elements consisting of $\Omega_j$ is uniformly bounded.
Each $\Omega_j$ is enlarged to form a larger region $\Omega_j'$ by adding layers of fine elements with the overlap width $\delta$.
If we set
\begin{equation*}
V_j = \left\{ v_j \in H_0^1 (\Omega_j'): v_j|_{e} \in \mathcal{P}_1 (e) \gap \forall e \in \cT_h \text{ inside } \Omega_j' \right\}, \quad 1 \leq j \leq \cN,
\end{equation*}
and
\begin{equation*}
V_0 = \left\{ v_0 \in H_0^1 (\Omega) : v_0|_{e} \in \mathcal{P}_1 (e) \gap \forall e \in \cT_H \right\},
\end{equation*}
then $\{ V_j \}_{j=0}^\cN$ forms a space decomposition of $V = V_h$, i.e.,
\begin{equation*}
V = \sum_{j=0}^\cN R_j^T V_j,
\end{equation*}
where $R_j^T \colon V_j \rightarrow V$, $0 \leq j \leq \cN$, is the natural interpolation operator.
In this setting, the standard two-level additive Schwarz preconditioner is given by
\begin{equation}
\label{prec_original}
M^{-1} = \sum_{j=0}^\cN R_j^T K_j^{-1} R_j,
\end{equation}
where $K_j \colon V_j \rightarrow V_j$, $0 \leq j \leq \cN$, is the local stiffness matrix on the subspace $V_j$, i.e., $K_j = R_j K R_j^T$.
The additive Schwarz condition number~(see~\eqref{kASM}) of the preconditioned operator $M^{-1}K$ satisfies the following upper bound~\cite[Theorem~3.13]{TW:2005}.

\begin{proposition}
\label{Prop:prec_original}
Let $M^{-1}$ be the two-level additive Schwarz preconditioner defined in~\eqref{prec_original}.
Then it satisfies
\begin{equation*}
\kappa_{\ASM} (M^{-1} K) \leq C \left( 1 + \frac{H}{\delta} \right),
\end{equation*}
where $\kappa_{\ASM}$ denotes the additive Schwarz condition number defined in~\eqref{kASM} and $C$ is a positive constant independent of $h$, $H$, and $\delta$.
\end{proposition}

Proposition~\ref{Prop:prec_original} can be proven by using the abstract convergence theory of additive Schwarz methods presented in~\cite{TW:2005,Park:2020}; see \ref{App:ASM} for a brief summary.
Combining Corollaries~\ref{Cor:ES} and~\ref{Cor:SSE} with Proposition~\ref{Prop:prec_original}, we deduce that the preconditioner $M^{-1}$ works for the ES-FEM and SSE as well as for the standard FEM.

\begin{corollary}
\label{Cor:prec_original}
Let $M^{-1}$ be the two-level additive Schwarz preconditioner defined in~\eqref{prec_original} and let $\bar{K}$ be the stiffness matrix of either the ES-FEM or SSE.
Then it satisfies that
\begin{equation*}
\kappa_{\ASM} (M^{-1} \bar{K}) \leq C \left( 1 + \frac{H}{\delta} \right),
\end{equation*}
where $\kappa_{\ASM}$ denotes the additive Schwarz condition number defined in~\eqref{kASM} and $C$ is a positive constant independent of $h$, $H$, and $\delta$.
\end{corollary}

Corollary~\ref{Cor:prec_original} means that preconditioning the ES-FEM and SSE by $M^{-1}$ is as advantageous as preconditioning the standard FEM by $M^{-1}$.
For instance, the $M^{-1}$-preconditioned SSE is scalable in the sense that its condition number does not deteriorate even if the fine mesh size $h$ decreases when the coarse mesh size $H$ and overlap width $\delta$ decrease keeping $H/\delta$ and $\delta/h$ constant.
Therefore, the $M^{-1}$-preconditioned SSE is suitable for large-scale parallel computing in a manner that each subspace $V_j$ is assigned to a processor; this aspect is a usual advantage of subspace correction methods as parallel numerical solvers~\cite{TW:2005}.

\subsection{Enhanced local problems}
Until now, we have observed that subspace correction preconditioners designed for the standard FEM perform their roles properly even if they are applied to either the ES-FEM or SSE.
Meanwhile, for each of the methods, preconditioners can be modified to be more suitable to the method to achieve better performance.
The idea is straightforward; we simply replace the local stiffness matrices in~\eqref{prec_original} defined in terms of the standard FEM with those defined in terms of either the ES-FEM or SSE.
Let $\bar{K} \colon V \rightarrow V$ be the stiffness matrix of either the ES-FEM or SSE.
We set
\begin{equation}
\label{prec_enhanced}
\bar{M}^{-1} = \sum_{j=0}^\cN R_j^T \bar{K}_j^{-1} R_j, 
\end{equation}
where $\bar{K}_j \colon V_j \rightarrow V_j$, $0 \leq j \leq \cN$, is defined by $\bar{K}_j = R_j \bar{K} R_j^T$.
That is, $\bar{K}_j$ is the local stiffness matrix of either the ES-FEM or SSE on the subspace $V_j$.
Invoking Theorem~\ref{Thm:ASM}, we can mathematically explain why the preconditioner $\bar{M}^{-1}$ performs better than $M^{-1}$ when it is applied to either the ES-FEM or SSE.
Theorem~\ref{Thm:enhanced} says that $\bar{M}^{-1}$ is a better preconditioner for $\bar{K}$ than $M^{-1}$, and that it inherits good properties of $M^{-1}$ such as the scalability.

\begin{theorem}
\label{Thm:enhanced}
The enhanced additive Schwarz preconditioner $\bar{M}^{-1}$ defined in~\eqref{prec_enhanced} performs better than the original additive Schwarz preconditioner $M^{-1}$ defined in~\eqref{prec_original}.
More precisely, it satisfies
\begin{equation*}
\kappa_{\ASM} (\bar{M}^{-1} \bar{K} ) \leq \kappa_{\ASM} (M^{-1} \bar{K}),
\end{equation*}
where $\bar{K}$ is the stiffness matrix of either the ES-FEM or SSE, and $\kappa_{\ASM}$ denotes the additive Schwarz condition number defined in~\eqref{kASM}.
\end{theorem}
\begin{proof}
By the definition of $\kappa_{\ASM}$, it suffices to show that
\begin{equation*}
\omega_0 (\bar{M}^{-1} \bar{K} ) C_0^2 (\bar{M}^{-1} \bar{K} ) \leq \omega_0 (M^{-1} \bar{K}) C_0^2 (M^{-1} \bar{K}),\quad
\tau_0 (\bar{M}^{-1} \bar{K} ) \geq \tau_0 (M^{-1} \bar{K}),
\end{equation*}
where $C_0$, $\tau_0$, and $\omega_0$ are defined in Assumptions~\ref{Ass:stable},~\ref{Ass:CS}, and~\ref{Ass:local}, respectively.
First, we readily get $\tau_0 (\bar{M}^{-1} \bar{K} ) = \tau_0 (M^{-1} \bar{K})$ because Assumption~\ref{Ass:CS} does not rely on which local operators are used.
Since $\bar{K}_j = R_j \bar{K} R_j^T$, $0 \leq j \leq \cN$, it follows by the definition of $\omega_0 $ that $\omega_0 (\bar{M}^{-1} \bar{K} ) = 1$.
Meanwhile, as the inequality
\begin{equation*}
v_j^T \bar{K}_j v_j \leq \lambda_{\max}(K_j^{-1} \bar{K}_j) v_j^T K_j v_j \quad \forall v_j \in V_j
\end{equation*}
is sharp for all $0 \leq j \leq \cN$, we get
\begin{equation*}
\omega_0 (M^{-1} \bar{K}) = \max_{0 \leq j \leq \cN} \lambda_{\max}(K_j^{-1} \bar{K}_j).
\end{equation*}

Next, we take any $v \in V$ and let $v = \sum_{j=0}^\cN R_j^T v_j$ be a decomposition of $v$ such that
\begin{equation*}
\sum_{j=0}^\cN v_j^T K_j v_j = C_0^2 (M^{-1} \bar{K}) v^T \bar{K} v.
\end{equation*}
We refer to~\cite[Lemma~2.5]{TW:2005} for the existence of this decomposition.
It follows that
\begin{equation*}
\sum_{j=0}^\cN v_j^T \bar{K}_j v_j
 \leq  \sum_{j=0}^\cN \lambda_{\max} (K_j^{-1} \bar{K}_j ) v_j^T K_j v_j
 \leq \omega_0 (M^{-1} \bar{K}) \sum_{j=0}^\cN v_j^T K_j v_j
 = \omega_0 (M^{-1} \bar{K}) C_0^2 (M^{-1} \bar{K}) v^T \bar{K}^T v,
\end{equation*}
which implies $C_0^2 (\bar{M}^{-1} \bar{K}) \leq \omega_0 (M^{-1} \bar{K}) C_0^2 (M^{-1} \bar{K})$ by the definition of $C_0$.
Consequently, we have
\begin{equation*}
\omega_0 (\bar{M}^{-1} \bar{K}) C_0^2 (\bar{M}^{-1} \bar{K})
= C_0^2 (\bar{M}^{-1} \bar{K}) \leq \omega_0 (M^{-1} \bar{K}) C_0^2 (M^{-1} \bar{K}),
\end{equation*}
which completes the proof.
\end{proof}

We conclude this section by introducing a variant of~\eqref{prec_enhanced} that is more convenient to implement.
When we implement~\eqref{prec_enhanced}, a cumbersome step is to assemble the coarse stiffness matrix $\bar{K}_0 = R_0 \bar{K} R_0^T$.
While the interpolation operator $R_0^T$ is defined for functions in the coarse space $V_0$, the strain smoothing procedure in $\bar{K}$ is defined in the fine-scale space $V$.
Hence, to assemble $\bar{K}_0$, fine-scale computations are required, although it acts on the coarse space $V_0$.
To explain in detail, we consider how to compute each entry of the coarse stiffness matrix $\bar{K}_0$.
The entry $(\bar{K}_0)_{ij}$ on the $i$th row and $j$th column is given by
\begin{equation*}
(\bar{K}_0)_{ij} = \intO \bar{\nabla} (R_0^T \phi_i) \cdot \bar{\nabla} (R_0^T \phi_j) \,d\Omega,
\end{equation*}
where $\bar{\nabla}$ is either $\bar{\nabla}_{\ES}$ or $\bar{\nabla}_{\SSE}$, and $\phi_i$ and $\phi_j$ denote the $i$th and $j$th nodal basis functions for $V_0$, respectively.
For simplicity, we suppose that $\bar{\nabla} = \bar{\nabla}_{\ES}$ and that $\cT_h$ is a refinement of $\cT_H$.
Since $R_0^T \phi_i$ is continuous and piecewise linear on $\cT_H$, $\nabla (R_0^T \phi_i)$ is contained in a coarse-scale space $W_H$, where
\begin{equation*}
W_H = \left\{ \epsilon \in (L^2 (\Omega))^2 : \epsilon|_{e} \in (\cP_0 (e))^2 \gap \forall e \in \cT_H \right\}.
\end{equation*}
In contrast, by the definition of $\bar{\nabla}_{\ES}$, $\bar{\nabla}_{\ES} (R_0^T \phi_i)$ is contained in a fine-scale space $W_{h, \ES}$.
Hence, we need integration on the fine mesh $\cT_h$ in order to compute $(\bar{K}_0)_{ij}$, while it suffices to perform integration on the coarse mesh $\cT_H$ to compute $(K_0)_{ij}$.

To avoid such fine-scale computations, we propose the following alternative two-level additive Schwarz preconditioner:
\begin{equation}
\label{prec_enhanced_alt}
\bar{M}_{\alt}^{-1} = R_0^T K_0^{-1} R_0 + \sum_{j=1}^\cN R_j^T \bar{K}_j^{-1} R_j .
\end{equation}
Note that if $\cT_h$ is a refinement of $\cT_H$, then $K_0$ agrees with the stiffness matrix of the standard FEM associated with the coarse mesh $\cT_H$.
The alternative preconditioner $\bar{M}_{\alt}^{-1}$ involves the stiffness matrices of the strain smoothing methods in the fine-scale subspaces, whereas its coarse-scale operation is defined in terms of the standard FEM.
Therefore, it is easier to implement than $\bar{M}^{-1}$.
Because $\bar{M}_{\alt}^{-1}$ is a type of hybrid of $M^{-1}$ and $\bar{M}^{-1}$, it is expected that the convergence behavior of the $\bar{M}_{\alt}^{-1}$-preconditioned operator lies between those of the $M^{-1}$- and $\bar{M}^{-1}$-preconditioned operators.
Numerical comparisons among the preconditioners $M^{-1}$, $\bar{M}^{-1}$, and $\bar{M}_{\alt}^{-1}$ will be presented in Section~\ref{Sec:Numerical}.

\section{Numerical results}
\label{Sec:Numerical}
In this section, we verify the theoretical results presented through numerical experiments.   
We solve two-dimensional Poisson and linear elasticity problems using three-node triangular elements with the ES-FEM and SSE. 
The preconditioned conjugate gradient method is used to solve a linear system $Au = f$, $A \in \{ K, \bar{K}_{\ES}, \bar{K}_{\SSE} \}$, with a stop criterion
\begin{equation*}
\frac{\| A u ^{(k)} - f \|_{\ell^2}}{\| f \|_{\ell^2}} < 10^{-12},
\end{equation*}
and with zero initial guess, where $\| \cdot \|_{\ell^2}$ denotes the $\ell^2$-norm of degrees of freedom and $u^{(k)}$ is the $k$th iterate of the preconditioned conjugate gradient method.
We verify the spectral equivalence among the stiffness matrices of the standard FEM, ES-FEM, and SSE.
Additionally, we compare the performance of the two-level additive Schwarz preconditioners $M^{-1}$, $\bar{M}^{-1}$, and $\bar{M}_{\alt}^{-1}$ introduced in~\eqref{prec_original},~\eqref{prec_enhanced}, and~\eqref{prec_enhanced_alt}, respectively.

\subsection{Poisson equation}

\begin{figure}[]
\centering
\includegraphics[width=1\textwidth]{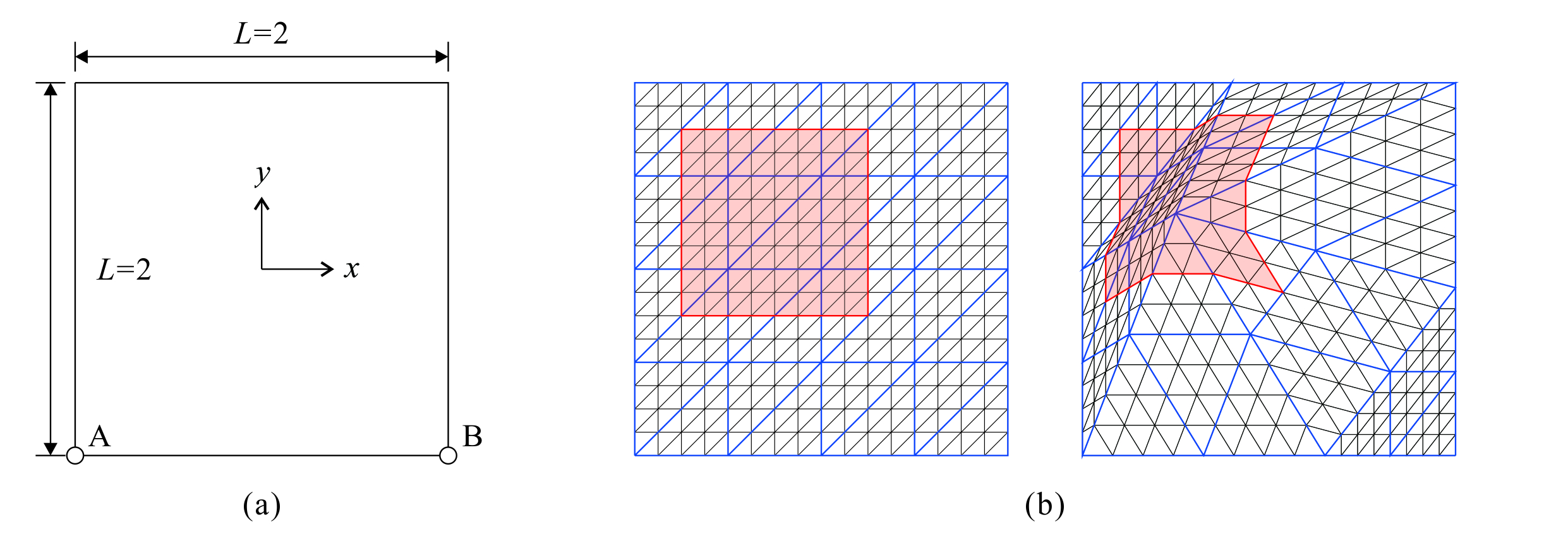}
\caption{\textbf{(a)} Square domain $\Omega = (-1,1)^2 \subset \mathbb{R}^2$ for the model Poisson and linear elasticity problems. \textbf{(b)} Domain decomposition settings for structured and unstructured meshes when $n=2^4, N=2^2$, and $\delta=2h$.
The blue and black lines represent the coarse and fine triangulations $\cT_H$ and $\cT_h$, respectively.}
\label{Fig:block}
\end{figure}

The first example is the Poisson problem~\eqref{model} defined on the domain $\Omega = (-1, 1)^2 \subset \mathbb{R}^2$ shown in Fig.~\ref{Fig:block}(a), where the function $f$ is given such that the problem has the exact solution $u(x,y)=e^{8(x+y)}\sin(\pi x)\sin(\pi y)$.
The side length of the square domain $\Omega$ is denoted by $L=2$.
We employ two types of coarse triangulations $\cT_H$: the standard checkerboard type structured triangulation and an unstructured triangulation with nonuniform nodal points, with $2 \times N \times N$ coarse elements~($N=2$, $2^2$, \dots, $2^5$).
Fine triangulations $\cT_h$ are constructed as the uniform refinements of $\cT_H$ such that there are $2 \times n \times n$ fine elements in the whole domain $\Omega$~($n=2^3$, $2^4$, \dots, $2^7$).
Each nonoverlapping subdomain $\Omega_j$, $1 \leq j \leq \cN = N \times N$, is a quadrilateral region composed of two coarse elements, and the corresponding overlapping subdomain $\Omega_j'$ is formed by adding $\delta/h$ layers of fine elements; see Fig.~\ref{Fig:block}(b) for the case of $n=2^4$ and $N=2^2$.
The characteristic sizes of the fine and coarse meshes are calculated by $h=L/n$ and $H=L/N$, respectively.
We set the overlap width $\delta$ as $h$ or $2h$, such that $\delta/h$ is constant.

\begin{table}[] \centering
\begin{tabular}{ccccccccccccc}
\hline
{} && {} && $n=2^3$ && $n=2^4$ && $n=2^5$ && $n=2^6$ && $n=2^7$ \\
\hline
\multirow{2}{*}{ES-FEM} && Structured mesh		&& 1.90e0 && 2.04e0 && 2.07e0 && 2.08e0 && 2.09e0 \\
                                     && Unstructured mesh	&& 2.21e0 && 2.60e0 && 2.88e0 && 3.05e0 && 3.12e0 \\
\hline
\multirow{2}{*}{SSE} && Structured mesh		&& 2.87e0 && 3.24e0 && 3.34e0 && 3.37e0 && 3.38e0 \\
                                && Unstructured mesh	&& 3.45e0 && 4.32e0 && 4.90e0 && 5.26e0 && 5.40e0 \\
\hline 
\end{tabular}
\caption{Condition numbers $\kappa ( K^{-1} \bar{K}_{\ES})$ and $\kappa ( K^{-1} \bar{K}_{\SSE})$ for the model Poisson problem~\eqref{model}.}
\label{Table:Spectrum_Poisson}
\end{table}

\begin{table}[] \centering
\resizebox{\textwidth}{!}{
\begin{tabular}{ccccccccccccccccc}
\hline
\multirow{2}{*}{Precond.} & \multirow{2}{*}{$N$} && \multicolumn{2}{c}{$n=2^3$} && \multicolumn{2}{c}{$n=2^4$} && \multicolumn{2}{c}{$n=2^5$} && \multicolumn{2}{c}{$n=2^6$} && \multicolumn{2}{c}{$n=2^7$} \\
\cline{4-5} \cline{7-8} \cline{10-11} \cline{13-14} \cline{16-17}
&&& \#iter & $\kappa$ && \#iter & $\kappa$ && \#iter & $\kappa$ && \#iter & $\kappa$ && \#iter & $\kappa$ \\ 
\hline
None &&& 17 & 1.73e1 && 34 & 6.92e1 && 69 & 2.77e2 && 138 & 1.11e3 && 278 & 4.43e3 \\
\hline
\multirow{5}{*}{\begin{tabular}{c}$M^{-1}$\end{tabular}}
& $2$ && 13 & 5.73e0 && 14 & 6.67e0 && 13 & 6.41e0 && 14 & 7.99e0 && 15 & 1.52e1 \\
& $2^2$ && - & - && 16 & 5.98e0 && 17 & 6.63e0 && 17 & 6.91e0 && 21 & 1.07e1 \\
& $2^3$ && - & - && - & - && 16 & 5.86e0 && 18 & 6.72e0 && 19 & 7.00e0 \\
& $2^4$ && - & - && - & - && - & - && 18 & 6.00e0 && 19 & 6.76e0 \\
& $2^5$ && - & - && - & - && - & - && - & - && 18 & 6.04e0 \\
\hline 
\multirow{5}{*}{\begin{tabular}{c}$\bar{M}^{-1}$\end{tabular}}
& $2$ && 9 & 4.81e0 && 9 & 5.58e0 && 10 & 6.84e0 && 11 & 1.00e1 && 12 & 1.76e1 \\
& $2^2$ && - & - && 14 & 5.45e0 && 14 & 5.88e0 && 15 & 6.98e0 && 20 & 1.20e1 \\
& $2^3$ && - & - && - & - && 15 & 5.49e0 && 15 & 5.91e0 && 17 & 7.00e0 \\
& $2^4$ && - & - && - & - && - & - && 16 & 5.50e0 && 16 & 5.91e0 \\
& $2^5$ && - & - && - & - && - & - && - & - && 16 & 5.45e0 \\
\hline 
\multirow{5}{*}{\begin{tabular}{c}$\bar{M}_{\alt}^{-1}$\end{tabular}}
& $2$ && 9 & 4.75e0 && 9 & 5.62e0 && 9 & 6.88e0 && 11 & 1.00e1 && 12 & 1.76e1 \\
& $2^2$ && - & - && 14 & 5.33e0 && 14 & 5.90e0 && 15 & 7.00e0 && 20 & 1.20e1 \\
& $2^3$ && - & - && - & - && 15 & 5.37e0 && 15 & 5.95e0 && 16 & 7.00e0 \\
& $2^4$ && - & - && - & - && - & - && 16 & 5.35e0 && 16 & 5.94e0 \\
& $2^5$ && - & - && - & - && - & - && - & - && 16 & 5.24e0 \\
\hline 
\end{tabular}
}
\caption{Condition numbers $\kappa$ and iteration counts \#iter of the ES-FEM applied to the model Poisson problem~\eqref{model} for the structured meshes and $\delta=2h$ with the two-level additive Schwarz preconditioners $M^{-1}$, $\bar{M}^{-1}$, and $\bar{M}_{\alt}^{-1}$ defined in~\eqref{prec_original},~\eqref{prec_enhanced}, and~\eqref{prec_enhanced_alt}, respectively.}
\label{Table:Poisson_ES-FEM}
\end{table}

\begin{table}[] \centering
\resizebox{\textwidth}{!}{
\begin{tabular}{ccccccccccccccccc}
\hline
\multirow{2}{*}{Precond.} & \multirow{2}{*}{$N$} && \multicolumn{2}{c}{$n=2^3$} && \multicolumn{2}{c}{$n=2^4$} && \multicolumn{2}{c}{$n=2^5$} && \multicolumn{2}{c}{$n=2^6$} && \multicolumn{2}{c}{$n=2^7$} \\
\cline{4-5} \cline{7-8} \cline{10-11} \cline{13-14} \cline{16-17}
&&& \#iter & $\kappa$ && \#iter & $\kappa$ && \#iter & $\kappa$ && \#iter & $\kappa$ && \#iter & $\kappa$ \\ 
\hline
None &&& 15 & 1.32e1 && 30 & 5.22e1 && 59 & 2.08e2 && 119 & 8.30e2 && 240 & 3.32e3 \\
\hline
\multirow{5}{*}{\begin{tabular}{c}$M^{-1}$\end{tabular}}
& $2$ && 14 & 7.16e0 && 16 & 9.49e0 && 16 & 9.55e0 && 15 & 9.30e0 && 16 & 1.40e1 \\
& $2^2$ && - & - && 18 & 7.62e0 && 20 & 9.53e0 && 20 & 1.01e1 && 21 & 1.01e1 \\
& $2^3$ && - & - && - & - && 19 & 7.72e0 && 20 & 9.47e0 && 22 & 1.03e1 \\
& $2^4$ && - & - && - & - && - & - && 20 & 7.55e0 && 22 & 9.59e0 \\
& $2^5$ && - & - && - & - && - & - && - & - && 21 & 7.74e0 \\
\hline 
\multirow{5}{*}{\begin{tabular}{c}$\bar{M}^{-1}$\end{tabular}}
& $2$ && 9 & 4.77e0 && 9 & 5.48e0 && 9 & 6.73e0 && 10 & 9.49e0 && 12 & 1.67e1 \\
& $2^2$ && - & - && 14 & 5.42e0 && 14 & 5.77e0 && 15 & 6.87e0 && 20 & 1.18e1 \\
& $2^3$ && - & - && - & - && 15 & 5.46e0 && 15 & 5.79e0 && 16 & 6.87e0 \\
& $2^4$ && - & - && - & - && - & - && 16 & 5.46e0 && 16 & 5.79e0 \\
& $2^5$ && - & - && - & - && - & - && - & - && 16 & 5.43e0 \\
\hline 
\multirow{5}{*}{\begin{tabular}{c}$\bar{M}_{\alt}^{-1}$\end{tabular}}
& $2$ && 9 & 4.69e0 && 9 & 5.53e0 && 9 & 6.79e0 && 10 & 9.49e0 && 12 & 1.67e1 \\
& $2^2$ && - & - && 14 & 5.27e0 && 14 & 5.80e0 && 15 & 6.90e0 && 20 & 1.18e1 \\
& $2^3$ && - & - && - & - && 15 & 5.30e0 && 15 & 5.85e0 && 16 & 6.88e0 \\
& $2^4$ && - & - && - & - && - & - && 16 & 5.29e0 && 16 & 5.84e0 \\
& $2^5$ && - & - && - & - && - & - && - & - && 16 & 5.23e0 \\
\hline    
\end{tabular}
}
\caption{Condition numbers $\kappa$ and iteration counts \#iter of the SSE applied to the model Poisson problem~\eqref{model} for the structured meshes and $\delta=2h$ with the two-level additive Schwarz preconditioners $M^{-1}$, $\bar{M}^{-1}$, and $\bar{M}_{\alt}^{-1}$ defined in~\eqref{prec_original},~\eqref{prec_enhanced}, and~\eqref{prec_enhanced_alt}, respectively.}
\label{Table:Poisson_SSE}
\end{table}

\begin{figure}[]
\centering
\includegraphics[width=1\textwidth]{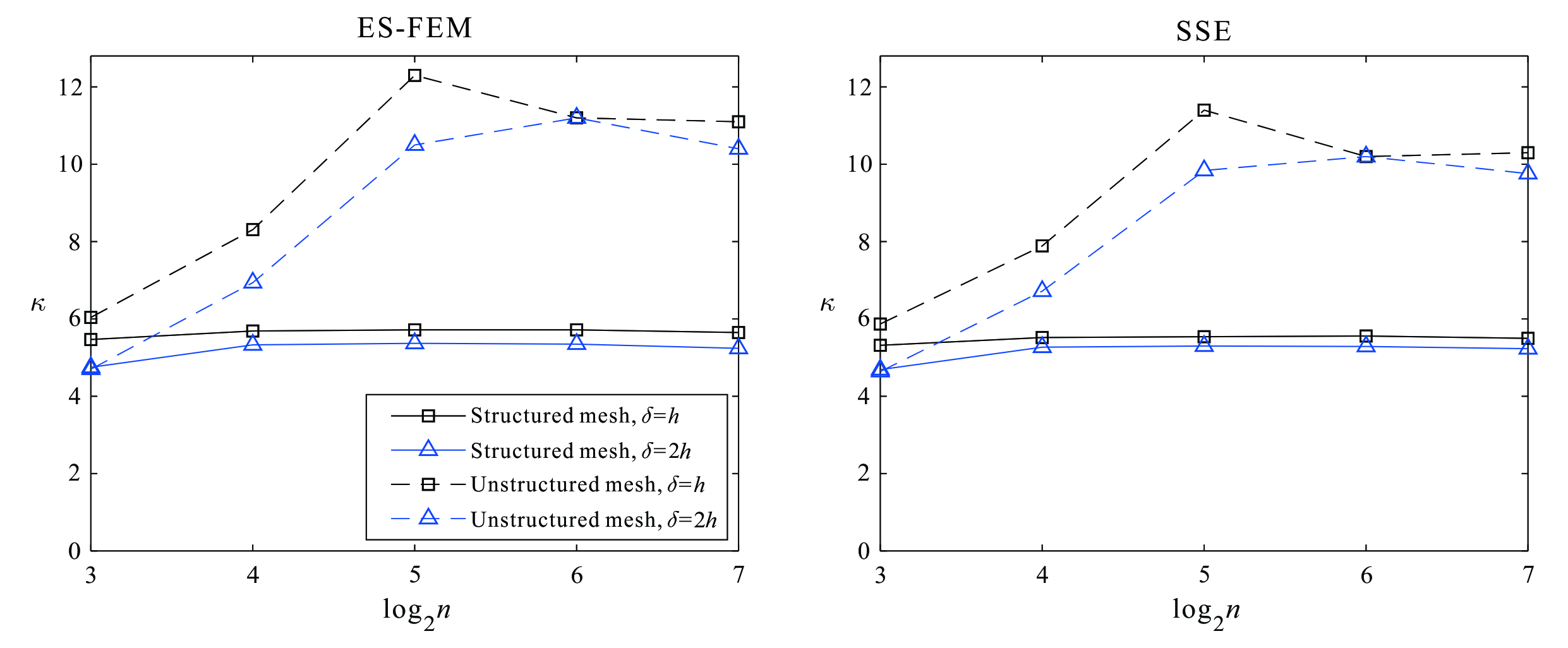}
\caption{Condition numbers $\kappa (\bar{M}_{\alt}^{-1} \bar{K})$ of the ES-FEM and SSE applied to the model Poisson problem~\eqref{model} when $n$ varies and $n/N$ $(=H/h)$ is fixed as $4$.}
\label{Fig:PCG_Poisson}
\end{figure}

\begin{figure}[]
\centering
\includegraphics[width=1\textwidth]{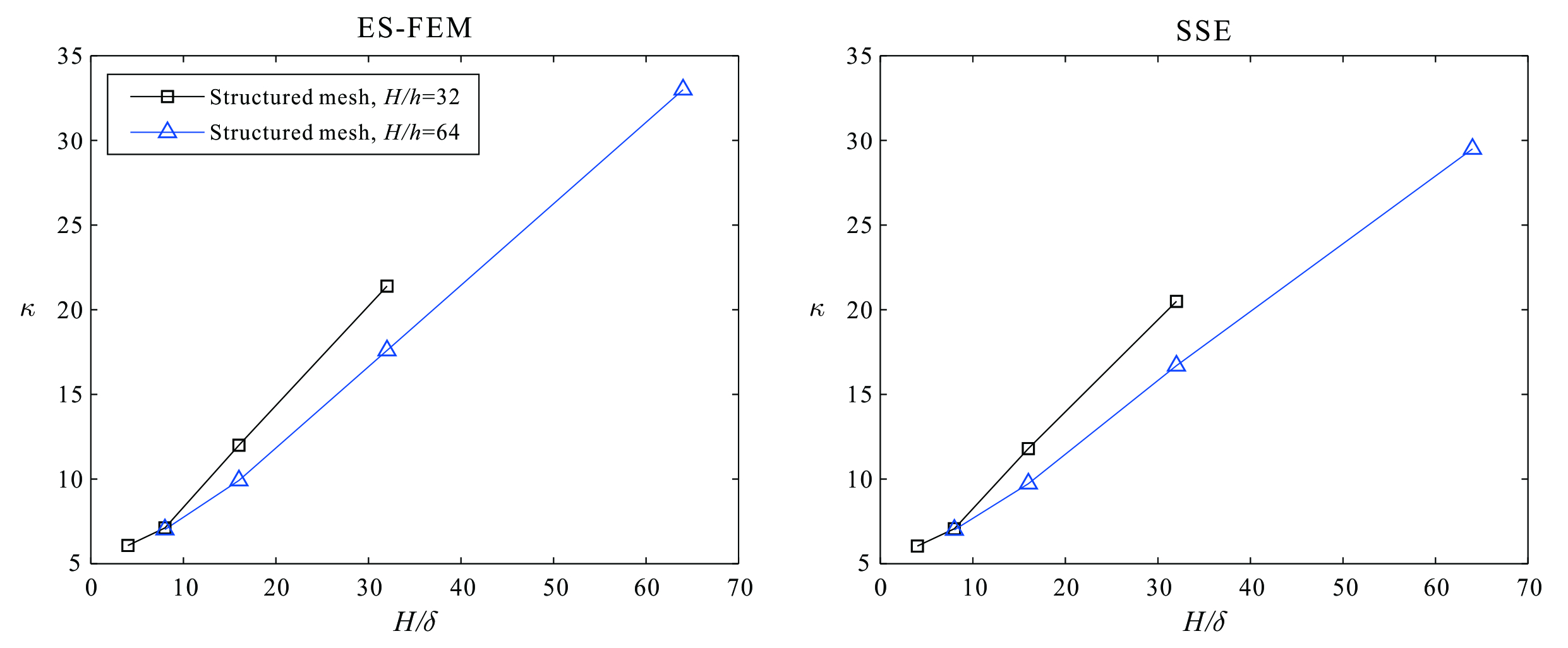}
\caption{Condition numbers $\kappa (\bar{M}_{\alt}^{-1} \bar{K})$ of the ES-FEM and SSE applied to the model Poisson problem~\eqref{model} when $H/\delta$ varies and $n/N$ $(=H/h)$ is fixed as $32$ or $64$.}
\label{Fig:Linear_Growth_Poisson}
\end{figure}

Table~\ref{Table:Spectrum_Poisson} provides the condition numbers $\kappa ( K^{-1} \bar{K}_{\ES})$ and $\kappa ( K^{-1} \bar{K}_{\SSE})$ for the structured and unstructured meshes for various values of $n$.
The condition numbers $\kappa ( K^{-1} \bar{K}_{\ES})$ and $\kappa ( K^{-1} \bar{K}_{\SSE})$ are eventually bounded when $n$ increases.
Hence, the stiffness matrices $\bar{K}_{\ES}$ and $\bar{K}_{\SSE}$ of the ES-FEM and SSE, respectively, are spectrally equivalent to the stiffness matrix $K$ of the standard FEM.
This numerically verifies Theorems~\ref{Thm:ES} and~\ref{Thm:SSE}.
Tables~\ref{Table:Poisson_ES-FEM} and~\ref{Table:Poisson_SSE} exhibit the condition numbers of the $M^{-1}$-, $\bar{M}^{-1}$-, and $\bar{M}_{\alt}^{-1}$-preconditioned stiffness matrices of the ES-FEM and SSE, respectively, and the corresponding conjugate gradient iteration counts denoted as \#iter for the structured meshes with various values of $n$ and $N$.
Fig.~\ref{Fig:PCG_Poisson} depicts the condition numbers of the $\bar{M}_{\alt}^{-1}$-preconditioned ES-FEM and SSE for the structured and unstructured meshes when $n/N$ $(=H/h)$ is fixed as $4$.
As we have explained theoretically in Section~\ref{Sec:Improvement}, $M^{-1}$-, $\bar{M}^{-1}$-, and $\bar{M}_{\alt}^{-1}$-preconditioned ES-FEM and SSE are all numerically scalable in the sense that the condition number and iteration count are eventually bounded when $n$ and $N$ increase keeping $n/N$ constant.
Moreover, we observe that each iteration count for the preconditioner $\bar{M}^{-1}$ is less than the corresponding counterpart for the preconditioner $M^{-1}$, which numerically verifies Theorem~\ref{Thm:enhanced}.
We also highlight that the condition numbers and iteration counts for the preconditioner $\bar{M}_{\alt}^{-1}$ are comparable to those for $\bar{M}^{-1}$.
Hence, as we have claimed in Section~\ref{Sec:Improvement}, $\bar{M}_{\alt}^{-1}$ can be a good alternative to $\bar{M}^{-1}$ with a comparable performance and easy implementation.
In addition, Fig.~\ref{Fig:Linear_Growth_Poisson} provides the condition numbers of the $\bar{M}_{\alt}^{-1}$-preconditioned ES-FEM and SSE for the structured meshes when $H/\delta$ varies and $n/N$ $(=H/h)$ is fixed as $32$ or $64$.
We verify the linear growth of the condition numbers for increasing values of $H/\delta$.


\subsection{Linear elasticity}
\label{Subsec:Elasticity}
We consider the following model linear elasticity problem defined on the domain $\Omega = (-1, 1)^2 \subset \mathbb{R}^2$:
\begin{equation}
\label{elasticity}
- \div \sigma (u) = b \gap\textrm{ in } \Omega,
\end{equation}
where $\sigma (u)$ is the Cauchy stress, $b$ is the body force given by $b (x,y) = (-y^2, 1-x^2)$, and the Dirichlet boundary condition $u=0$ is given along line AB~(see Fig.~\ref{Fig:block}(a)). The plane stress condition is employed with Young's modulus $E = 1 \times 10^3$ and Poisson's ratio $\nu = 0.2$.
The finite element models are constructed using $2 \times n \times n$ fine elements~($n=2^3$, $2^4$, \dots, $2^7$) and $ 2\times N \times N$ coarse elements~($N=2$, $2^2$, \dots, $2^5$) with the overlap width $\delta=h$ or $2h$, as shown in Fig.~\ref{Fig:block}(b).
We only present the results for the structured meshes; as in the case of the Poisson's equation, the unstructured meshes show similar behaviors as the structured meshes.

\begin{table}[] \centering
\begin{tabular}{ccccccccccc}
\hline
{} && $n=2^3$ && $n=2^4$ && $n=2^5$ && $n=2^6$ && $n=2^7$ \\
\hline
ES-FEM && 2.27e0 && 2.30e0 && 2.32e0 && 2.33e0 && 2.33e0 \\	
\hline
SSE && 3.68e0 && 3.78e0 && 3.84e0 && 3.86e0 && 3.87e0 \\
\hline 
\end{tabular}
\caption{Condition numbers $\kappa ( K^{-1} \bar{K}_{\ES})$ and $\kappa ( K^{-1} \bar{K}_{\SSE})$ with the structured meshes for the model linear elasticity problem~\eqref{elasticity}.}
\label{Table:Spectrum_LE}
\end{table}


\begin{table}[] \centering
\resizebox{\textwidth}{!}{
\begin{tabular}{ccccccccccccccccc}
\hline
\multirow{2}{*}{Precond.} & \multirow{2}{*}{$N$} && \multicolumn{2}{c}{$n=2^3$} && \multicolumn{2}{c}{$n=2^4$} && \multicolumn{2}{c}{$n=2^5$} && \multicolumn{2}{c}{$n=2^6$} && \multicolumn{2}{c}{$n=2^7$} \\
\cline{4-5} \cline{7-8} \cline{10-11} \cline{13-14} \cline{16-17}
&&& \#iter & $\kappa$ && \#iter & $\kappa$ && \#iter & $\kappa$ && \#iter & $\kappa$ && \#iter & $\kappa$ \\ 
\hline
None &&& 71 & 1.05e3 && 139 & 3.71e3 && 272 & 1.38e4 && 535 & 5.32e4 && 1065 & 2.08e5 \\
\hline
\multirow{5}{*}{\begin{tabular}{c}$M^{-1}$\end{tabular}}
& $2$ && 21 & 8.80e0 && 22 & 9.73e0 && 24 & 1.16e1 && 30 & 1.76e1 && 40 & 3.14e1 \\
& $2^2$ && - & - && 23 & 8.96e0 && 25 & 9.20e0 && 28 & 1.17e1 && 37 & 2.01e1 \\
& $2^3$ && - & - && - & - && 23 & 8.33e0 && 25 & 9.20e0 && 28 & 1.19e1 \\
& $2^4$ && - & - && - & - && - & - && 23 & 8.24e0 && 25 & 9.32e0 \\
& $2^5$ && - & - && - & - && - & - && - & - && 23 & 8.16e0 \\
\hline 
\multirow{5}{*}{\begin{tabular}{c}$\bar{M}^{-1}$\end{tabular}}
& $2$ && 18 & 7.66e0 && 20 & 9.26e0 && 22 & 1.22e1 && 28 & 1.94e1 && 38 & 3.51e1 \\
& $2^2$ && - & - && 20 & 6.88e0 && 21 & 7.58e0 && 26 & 1.18e1 && 33 & 2.13e1 \\
& $2^3$ && - & - && - & - && 20 & 6.83e0 && 22 & 8.58e0 && 26 & 1.19e1 \\
& $2^4$ && - & - && - & - && - & - && 20 & 6.72e0 && 22 & 8.78e0 \\
& $2^5$ && - & - && - & - && - & - && - & - && 20 & 6.71e0 \\
\hline       
\multirow{5}{*}{\begin{tabular}{c}$\bar{M}_{\alt}^{-1}$\end{tabular}}
& $2$ && 18 & 8.52e0 && 20 & 1.00e1 && 24 & 1.26e1 && 29 & 1.95e1 && 39 & 3.52e1 \\
& $2^2$ && - & - && 21 & 7.77e0 && 23 & 9.25e0 && 28 & 1.24e1 && 36 & 2.17e1 \\
& $2^3$ && - & - && - & - && 21 & 7.63e0 && 23 & 9.23e0 && 28 & 1.25e1 \\
& $2^4$ && - & - && - & - && - & - && 21 & 7.60e0 && 24 & 9.50e0 \\
& $2^5$ && - & - && - & - && - & - && - & - && 21 & 7.61e0 \\
\hline      
\end{tabular}
}
\caption{Condition numbers $\kappa$ and iteration counts \#iter of the ES-FEM applied to the model linear elasticity problem~\eqref{elasticity} for the structured meshes and $\delta=2h$ with the two-level additive Schwarz preconditioners $M^{-1}$, $\bar{M}^{-1}$, and $\bar{M}_{\alt}^{-1}$ defined in~\eqref{prec_original},~\eqref{prec_enhanced}, and~\eqref{prec_enhanced_alt}, respectively.}
\label{Table:LE_ES-FEM}
\end{table}

\begin{table}[] \centering
\resizebox{\textwidth}{!}{
\begin{tabular}{ccccccccccccccccc}
\hline
\multirow{2}{*}{Precond.} & \multirow{2}{*}{$N$} && \multicolumn{2}{c}{$n=2^3$} && \multicolumn{2}{c}{$n=2^4$} && \multicolumn{2}{c}{$n=2^5$} && \multicolumn{2}{c}{$n=2^6$} && \multicolumn{2}{c}{$n=2^7$} \\
\cline{4-5} \cline{7-8} \cline{10-11} \cline{13-14} \cline{16-17}
&&& \#iter & $\kappa$ && \#iter & $\kappa$ && \#iter & $\kappa$ && \#iter & $\kappa$ && \#iter & $\kappa$ \\ 
\hline
None &&& 67 & 8.00e2 && 123 & 2.80e3 && 237 & 1.04e4 && 464 & 3.99e4 && 922 & 1.56e5 \\
\hline
\multirow{5}{*}{\begin{tabular}{c}$M^{-1}$\end{tabular}}
& $2$ && 24 & 1.29e1 && 26 & 1.23e1 && 28 & 1.28e1 && 30 & 1.70e1 && 38 & 3.03e1 \\
& $2^2$ && - & - && 27 & 1.43e1 && 28 & 1.32e1 && 29 & 1.36e1 && 36 & 1.96e1 \\
& $2^3$ && - & - && - & - && 27 & 1.41e1 && 29 & 1.35e1 && 30 & 1.39e1 \\
& $2^4$ && - & - && - & - && - & - && 27 & 1.33e1 && 30 & 1.36e1 \\
& $2^5$ && - & - && - & - && - & - && - & - && 26 & 1.11e1 \\
\hline 
\multirow{5}{*}{\begin{tabular}{c}$\bar{M}^{-1}$\end{tabular}}
& $2$ && 18 & 7.45e0 && 19 & 9.02e0 && 22 & 1.20e1 && 28 & 1.90e1 && 36 & 3.43e1 \\
& $2^2$ && - & - && 20 & 6.76e0 && 21 & 7.34e0 && 25 & 1.15e1 && 33 & 2.08e1 \\
& $2^3$ && - & - && - & - && 20 & 6.67e0 && 22 & 8.35e0 && 25 & 1.16e1 \\
& $2^4$ && - & - && - & - && - & - && 20 & 6.55e0 && 22 & 8.55e0 \\
& $2^5$ && - & - && - & - && - & - && - & - && 20 & 6.54e0 \\
\hline       
\multirow{5}{*}{\begin{tabular}{c}$\bar{M}_{\alt}^{-1}$\end{tabular}}
& $2$ && 18 & 8.33e0 && 20 & 9.82e0 && 23 & 1.24e1 && 29 & 1.91e1 && 37 & 3.44e1 \\
& $2^2$ && - & - && 21 & 7.56e0 && 23 & 9.00e0 && 27 & 1.21e1 && 35 & 2.11e1 \\
& $2^3$ && - & - && - & - && 21 & 7.25e0 && 23 & 8.99e0 && 27 & 1.22e1 \\
& $2^4$ && - & - && - & - && - & - && 21 & 7.31e0 && 23 & 9.19e0 \\
& $2^5$ && - & - && - & - && - & - && - & - && 21 & 7.32e0 \\
\hline           
\end{tabular}
}
\caption{Condition numbers $\kappa$ and iteration counts \#iter of the SSE applied to the model linear elasticity problem~\eqref{elasticity} for the structured meshes and $\delta=2h$ with the two-level additive Schwarz preconditioners $M^{-1}$, $\bar{M}^{-1}$, and $\bar{M}_{\alt}^{-1}$ defined in~\eqref{prec_original},~\eqref{prec_enhanced}, and~\eqref{prec_enhanced_alt}, respectively.}
\label{Table:LE_SSE}
\end{table}

\begin{figure}[]
\centering\includegraphics[width=1\textwidth]{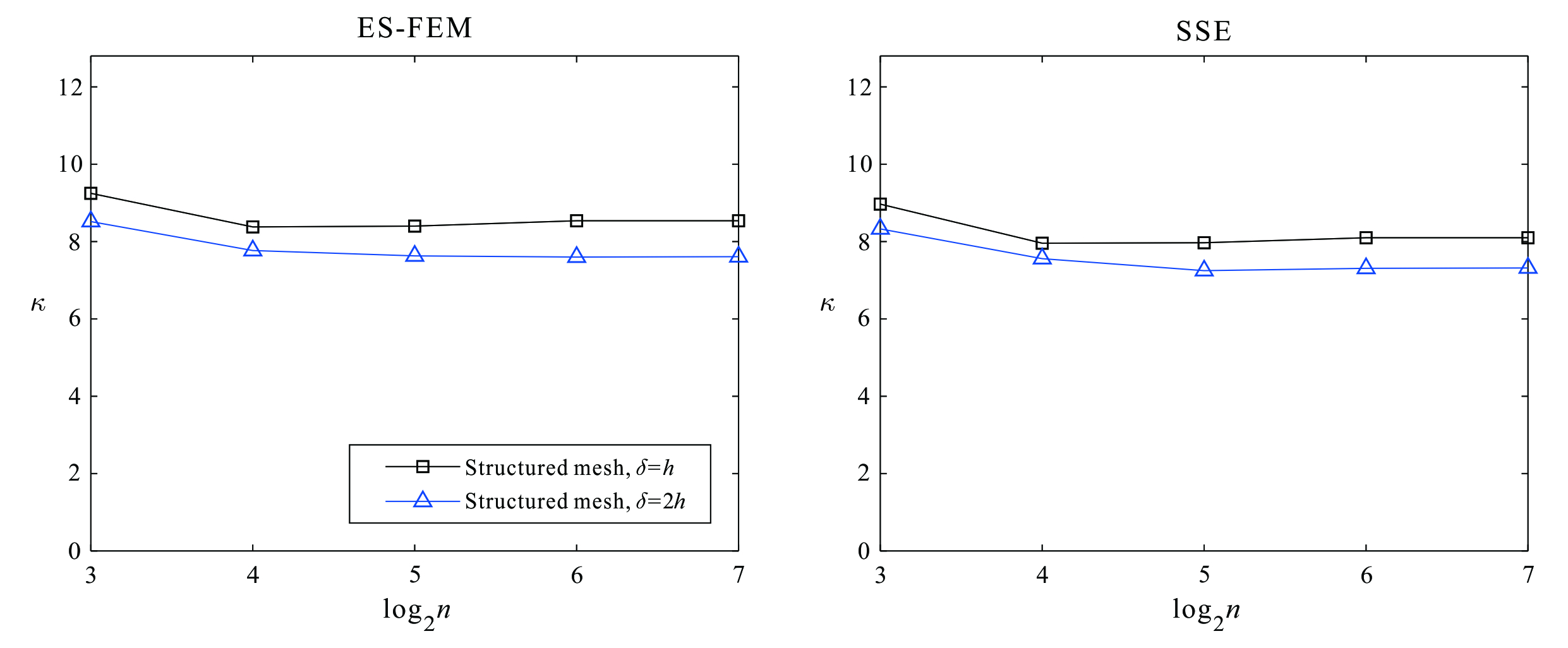}
\caption{Condition numbers $\kappa (\bar{M}_{\alt}^{-1} \bar{K})$ of the ES-FEM and SSE applied to the model linear elasticity problem~\eqref{elasticity} for the structured meshes when $n$ varies and $n/N$ $(=H/h)$ is fixed as $4$.}
\label{Fig:PCG_LE}
\end{figure}

Table~\ref{Table:Spectrum_LE} provides the condition numbers $\kappa ( K^{-1} \bar{K}_{\ES})$ and $\kappa ( K^{-1} \bar{K}_{\SSE})$ for various $n$.
Since the condition numbers for both ES-FEM and SSE are eventually bounded when $n$ becomes larger, we confirm the spectral equivalence among the stiffness matrices of the standard FEM, ES-FEM, and SSE.
Table~\ref{Table:LE_ES-FEM} presents the condition numbers of the $M^{-1}$-, $\bar{M}^{-1}$-, and $\bar{M}_{\alt}^{-1}$-preconditioned stiffness matrices and the corresponding conjugate gradient iteration counts \#iter for the ES-FEM.
Table~\ref{Table:LE_SSE} presents the results corresponding to the SSE.
Fig.~\ref{Fig:PCG_LE} shows the condition numbers when the preconditioner $\bar{M}_{\alt}^{-1}$ is used and $n/N$~$(=H/h)$ is fixed as $4$ for the ES-FEM and SSE.
Similar to the Poisson problem, it is observed that both the condition number and iteration count are eventually bounded when $n$ and $N$ increase, keeping the ratio $n/N$ constant, which implies that all the preconditioned methods are numerically scalable.
Moreover, the iteration counts for the preconditioners $\bar{M}^{-1}$ and $\bar{M}_{\alt}^{-1}$ are smaller than the corresponding values for $M^{-1}$ for most of the cases; this indicates the numerical efficiency of the proposed enhanced preconditioners $\bar{M}^{-1}$ and $\bar{M}_{\alt}^{-1}$ applied to the linear elasticity problem.
In conclusion, we have numerically proven that all the theoretical results developed in this study are valid for the Poisson and linear elasticity problems.

\section{Remarks on node-based strain smoothing}
\label{Sec:NS} 
We observed that the ES-FEM and SSE enjoy the spectral equivalence with the standard FEM.
In contrast, not every FEM with strain smoothing satisfies such equivalence property.
Particularly, we present an example in which the stiffness matrix of the NS-FEM may not be spectrally equivalent to that of the standard FEM.

In the NS-FEM, each element in $\cT_h$ is divided into three quadrilateral subdomains using its nodes, midpoints of element edges, and barycenter. The node-based smoothing domains consist of assemblages of adjacent subdomains belonging to different elements based on nodes; see~Fig.~\ref{Fig:domains}(c).
Denoting the collection of all smoothing domains constructed from $\cT_h$ for the NS-FEM by $\cS_{h,\NS}$, a smoothed gradient $\bar\epsilon^{(s)}$ for a smoothing domain $s \in \cS_{h,\NS}$ is defined by
\begin{equation}
\label{NS_gradient}
\bar\epsilon^{(s)} = \frac{\sum_{k=1}^{m} |e^{(k)}| \epsilon^{(e^{(k)})}}{\sum_{k=1}^{m} |e^{(k)}|},
\end{equation}
where $e^{(k)}$ is the $k$th element in $\cT_h$ neighboring to the node corresponding to $s$, $\epsilon^{(e^{(k)})}$ was defined in~\eqref{local_gradient}, and $m$ is the number of neighboring elements in $\cT_h$.
The number $m$ varies per node in general.
Using the smoothed gradient in~\eqref{NS_gradient}, the stiffness matrix $\bar{K}_{\NS}$ for the NS-FEM is defined  in a similar manner as~\eqref{stiffness_ES}.
It is known that the NS-FEM is effective in alleviating volumetric locking~\cite{LNNL:2009}.

The following example shows that the NS-FEM in one-dimension is not spectrally equivalent to the standard FEM; examples corresponding to higher-dimensional cases can be constructed similarly.

\begin{figure}[]
\centering
\includegraphics[width=1\textwidth]{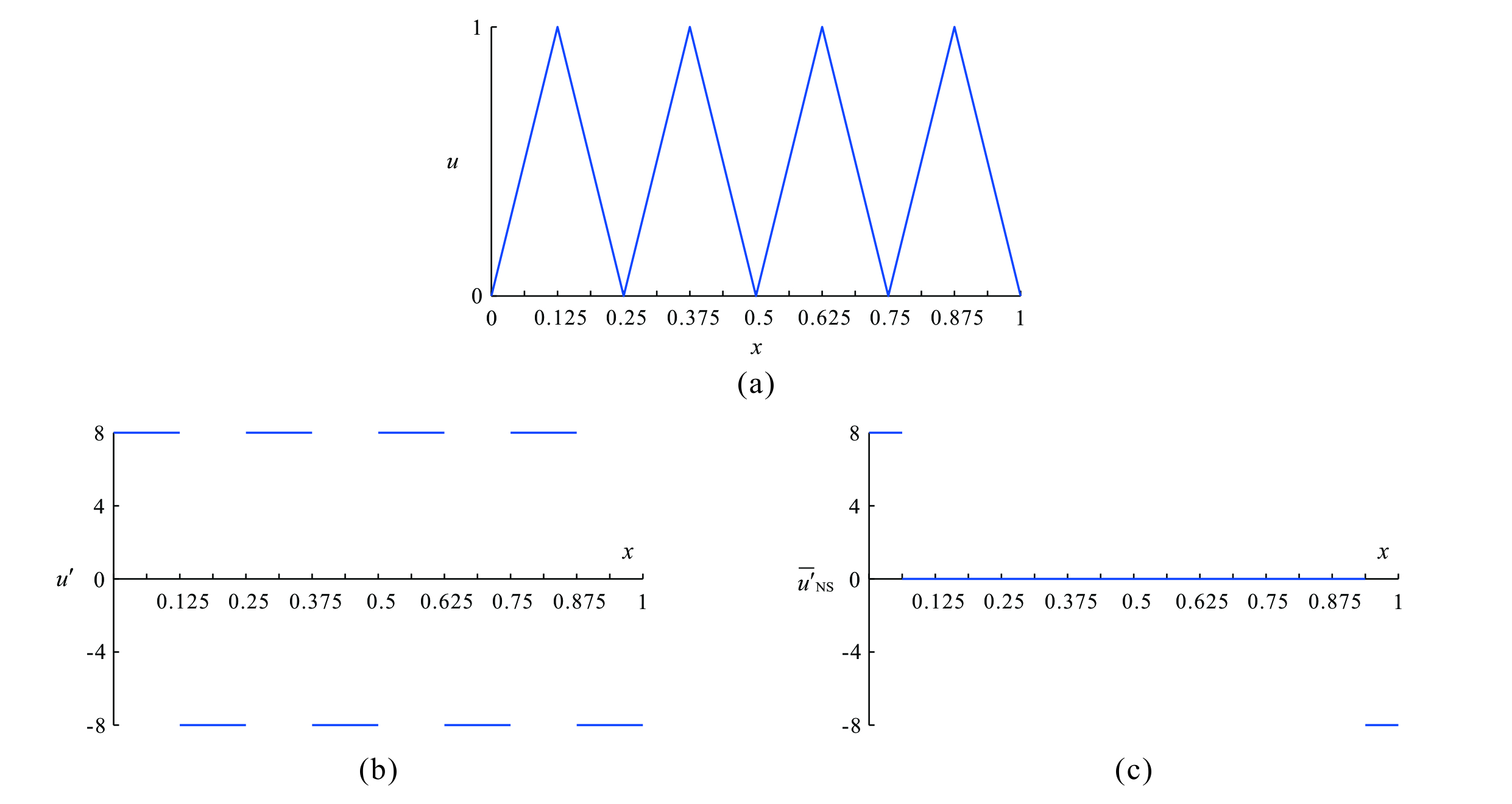}
\caption{Graphs of \textbf{(a)} the function $u$, \textbf{(b)} its derivative $u'$, and \textbf{(c)} the node-based smoothed derivative $\bar{u'}_{\NS}$ in Example~\ref{Ex:NS}, when $n=8$.}
\label{Fig:ex_ns_u}
\end{figure}

\begin{figure}[]
\centering
\includegraphics[width=1\textwidth]{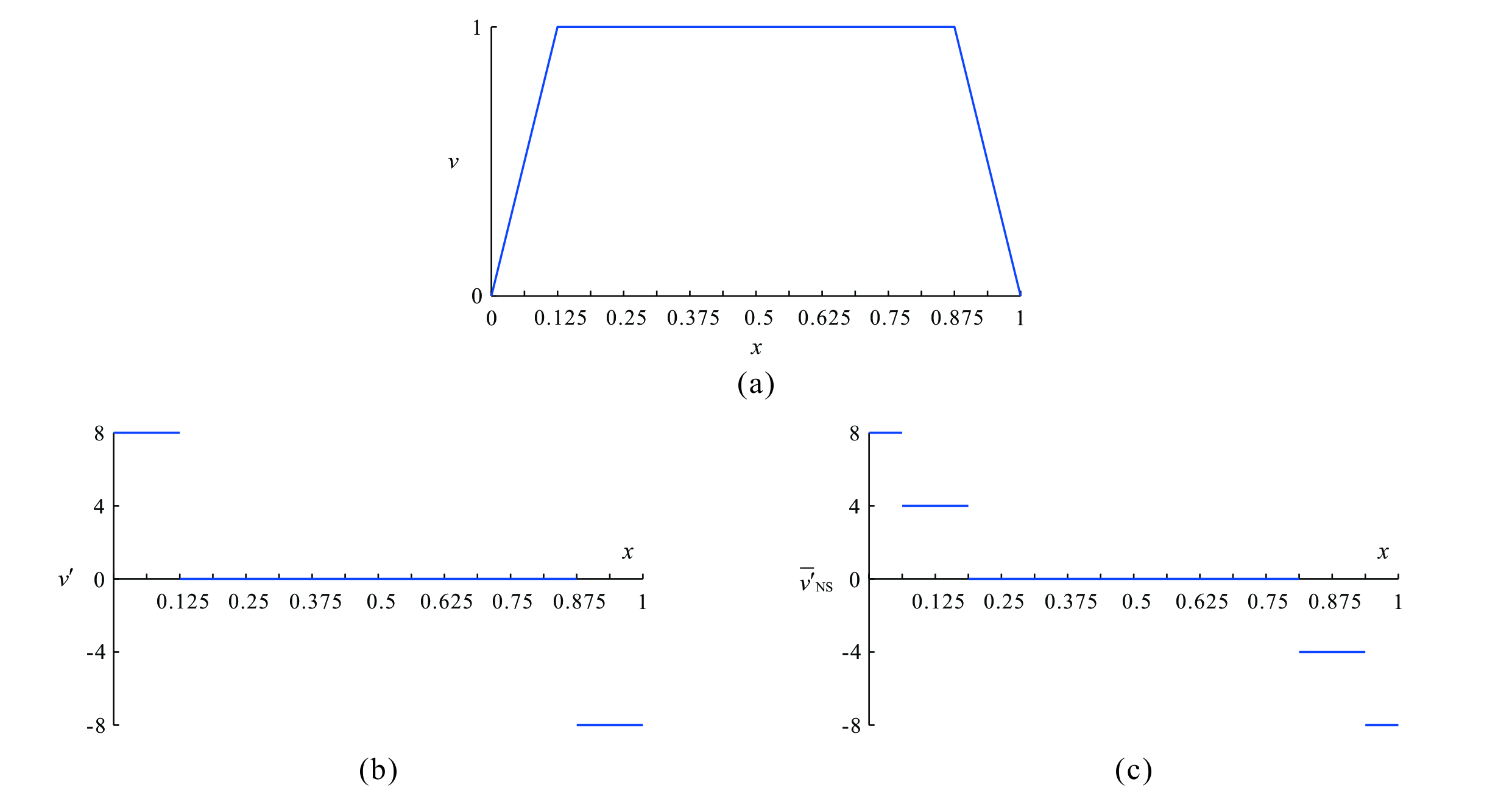}
\caption{Graphs of \textbf{(a)} the function $v$, \textbf{(b)} its derivative $v'$, and \textbf{(c)} the node-based smoothed derivative $\bar{v'}_{\NS}$ in Example~\ref{Ex:NS}, when $n=8$.}
\label{Fig:ex_ns_v}
\end{figure}

\begin{example}
\label{Ex:NS}
Let $\Omega = [0,1] \subset \mathbb{R}$ and let $\cT_h$ be the uniform partition of $\Omega$ into $n$ subintervals, where $n$ is a positive even integer.
In this case, the space $V_h$ is defined as the collection of all piecewise linear and continuous functions on $\cT_h$ satisfying the homogeneous Dirichlet boundary condition. 
As depicted in Fig.~\ref{Fig:ex_ns_u}(a), we set $u \in V_h$ such that
\begin{equation*}
u \left( \frac{i}{n} \right) = \begin{cases} 0 & \text{ if } i \text{ is even}, \\
1 & \text{ if } i \text{ is odd}, \end{cases}
\quad i = 1, 2, \dots, n-1.
\end{equation*}
In each subinterval $i/n < x < (i+1)/n$, $i=0, 1, \dots, n-1$, we have
\begin{equation*}
u'(x) = \begin{cases} n & \text{ if } i \text{ is even}, \\
-n & \text{ if } i \text{ is odd}. \end{cases}
\end{equation*}
By applying the node-based smoothing to $u'$, we obtain the smoothed derivative $\bar{u'}_{\NS}$ as follows:
\begin{equation*}
\bar{u'}_{\NS}(x) = \begin{cases} n & \text{ if } 0< x < \frac{1}{2n}, \\
0 & \text{ if } \frac{1}{2n} < x < 1 - \frac{1}{2n}, \\
-n & \textrm{ if } 1 - \frac{1}{2n} < x < 1. \end{cases}
\end{equation*}
The graphs of $u'$ and $\bar{u'}_{\NS}$ are plotted in Figs.~\ref{Fig:ex_ns_u}(b) and~(c), respectively.
Hence, it follows that
\begin{equation}
\label{NS_min}
\lambda_{\min}(K^{-1} \bar{K}_{\NS}) \leq \frac{u^T \bar{K}_{\NS} u}{u^T K u}
= \frac{\intO |\bar{u'}_{\NS}|^2 \,dx}{\intO |u'|^2 \,dx}
= \frac{1}{n}.
\end{equation}
Meanwhile, as shown in Fig.~\ref{Fig:ex_ns_v}(a), we set $v \in V_h$ such that
\begin{equation*}
v \left( \frac{i}{n} \right) = 1, \quad i=1,2, \dots, n-1.
\end{equation*}
Then one can readily obtain
\begin{equation*}
v'(x) = \begin{cases} n & \text{ if } 0< x < \frac{1}{n}, \\
0 & \text{ if } \frac{1}{n} < x < 1 - \frac{1}{n}, \\
-n & \textrm{ if } 1 - \frac{1}{n} < x < 1, \end{cases}
\end{equation*}
and
\begin{equation*}
\bar{v'}_{\NS}(x) = \begin{cases} n & \text{ if } 0< x < \frac{1}{2n}, \\
\frac{n}{2} & \text{ if } \frac{1}{2n} < x < \frac{3}{2n}, \\
0 & \text{ if } \frac{3}{2n} < x < 1 - \frac{3}{2n}, \\
-\frac{n}{2} & \text{ if } 1 - \frac{3}{2n} < x < 1 - \frac{1}{2n}, \\
-n & \textrm{ if } 1 - \frac{1}{2n} < x < 1. \end{cases}
\end{equation*}
See Figs.~\ref{Fig:ex_ns_v}(b) and~(c) for the graphs of $v'$ and $\bar{v'}_{\NS}$, respectively.
By direct calculation, we have
\begin{equation}
\label{NS_max}
\lambda_{\max} (K^{-1} \bar{K}_{\NS}) \geq \frac{v^T \bar{K}_{\NS} v}{v^T K v}
= \frac{\intO |\bar{v'}_{\NS}|^2 \,dx}{\intO |v'|^2 \,dx}
= \frac{3}{4}.
\end{equation}
Combining~\eqref{NS_min} and~\eqref{NS_max} yields
\begin{equation*}
\kappa ( K^{-1} \bar{K}_{\NS}) = \frac{\lambda_{\max} ( K^{-1} \bar{K}_{\NS})}{\lambda_{\min} ( K^{-1} \bar{K}_{\NS})} \geq \frac{3n}{4} = O( h^{-1}),
\end{equation*}
which implies that $\bar{K}_{\NS}$ and $K$ are not spectrally equivalent.

\begin{table}[] \centering
\begin{tabular}{ccccccccccc}
\hline
{} && $n=2^3$ && $n=2^4$ && $n=2^5$ && $n=2^6$ && $n=2^7$ \\
\hline
NS-FEM && 9.36e1 && 3.75e2 && 1.53e3 && 6.21e3 && 2.50e4 \\		
\hline 
\end{tabular}
\caption{Condition numbers $\kappa ( K^{-1} \bar{K}_{\NS})$ with the structured meshes for the model linear elasticity problem~\eqref{elasticity}.}
\label{Table:Spectrum_LE_NS-FEM}
\end{table}

\begin{table}[] \centering
\resizebox{\textwidth}{!}{
\begin{tabular}{ccccccccccccccccc}
\hline
\multirow{2}{*}{Precond.} & \multirow{2}{*}{$N$} && \multicolumn{2}{c}{$n=2^3$} && \multicolumn{2}{c}{$n=2^4$} && \multicolumn{2}{c}{$n=2^5$} && \multicolumn{2}{c}{$n=2^6$} && \multicolumn{2}{c}{$n=2^7$} \\
\cline{4-5} \cline{7-8} \cline{10-11} \cline{13-14} \cline{16-17}
&&& \#iter & $\kappa$ && \#iter & $\kappa$ && \#iter & $\kappa$ && \#iter & $\kappa$ && \#iter & $\kappa$ \\ 
\hline
None &&& 60 & 4.03e2 && 112 & 1.40e3 && 204 & 5.19e3 && 393 & 1.99e4 && 743 & 7.81e4 \\
\hline
\multirow{5}{*}{\begin{tabular}{c}$M^{-1}$\end{tabular}}
& $2$ && 60 & 1.22e2 && 133 & 7.81e2 && 270 & 3.87e3 && 498 & 1.77e4 && 906 & 7.62e4 \\
& $2^2$ && - & - && 116 & 4.47e2 && 269 & 3.26e3 && 515 & 1.58e4 && 943 & 6.54e4 \\
& $2^3$ && - & - && - & - && 218 & 1.82e3 && 496 & 1.30e4 && 963 & 6.24e4 \\
& $2^4$ && - & - && - & - && - & - && 389 & 7.25e3 && 906 & 5.14e4 \\
& $2^5$ && - & - && - & - && - & - && - & - && 694 & 2.69e4 \\
\hline 
\multirow{5}{*}{\begin{tabular}{c}$\bar{M}^{-1}$\end{tabular}}
& $2$ && 18 & 6.94e0 && 23 & 1.07e1 && 31 & 2.12e1 && 44 & 4.37e1 && 62 & 8.98e1 \\
& $2^2$ && - & - && 29 & 1.80e1 && 37 & 3.70e1 && 54 & 8.23e1 && 78 & 1.76e2 \\
& $2^3$ && - & - && - & - && 52 & 6.82e1 && 66 & 1.48e2 && 97 & 3.40e2 \\
& $2^4$ && - & - && - & - && - & - && 95 & 2.74e2 && 119 & 6.03e2 \\
& $2^5$ && - & - && - & - && - & - && - & - && 172 & 1.11e3 \\
\hline       
\multirow{5}{*}{\begin{tabular}{c}$\bar{M}_{\alt}^{-1}$\end{tabular}}
& $2$ && 18 & 7.85e0 && 24 & 1.10e1 && 32 & 2.16e1 && 45 & 4.39e1 && 63 & 8.99e1 \\
& $2^2$ && - & - && 29 & 1.82e1 && 38 & 3.97e1 && 56 & 8.57e1 && 80 & 1.78e2 \\
& $2^3$ && - & - && - & - && 52 & 6.93e1 && 69 & 1.62e2 && 102 & 3.59e2 \\
& $2^4$ && - & - && - & - && - & - && 96 & 2.80e2 && 125 & 6.63e2 \\
& $2^5$ && - & - && - & - && - & - && - & - && 173 & 1.13e3 \\
\hline      
\end{tabular}
}
\caption{Condition numbers $\kappa$ and iteration counts \#iter of the NS-FEM applied to the model linear elasticity problem~\eqref{elasticity} for the structured meshes and $\delta=2h$ with the two-level additive Schwarz preconditioners $M^{-1}$, $\bar{M}^{-1}$, and $\bar{M}_{\alt}^{-1}$ defined in~\eqref{prec_original},~\eqref{prec_enhanced}, and~\eqref{prec_enhanced_alt}, respectively.}
\label{Table:LE_NS-FEM}
\end{table}

We revisit the linear elasticity problem in Section~\ref{Subsec:Elasticity} using the structured meshes, as shown in Fig.~\ref{Fig:block}; Table~\ref{Table:Spectrum_LE_NS-FEM} provides the condition numbers $\kappa ( K^{-1} \bar{K}_{\NS})$ for various values of $n$.
The results show that the stiffness matrix of the NS-FEM is not spectrally equivalent to that of the standard FEM; $\kappa ( K^{-1} \bar{K}_{\NS})$ increases approximately four times whenever $n$ doubles.
Table~\ref{Table:LE_NS-FEM} presents the condition numbers of the $M^{-1}$-, $\bar{M}^{-1}$-, and $\bar{M}_{\alt}^{-1}$-preconditioned stiffness matrices and the corresponding conjugate gradient iteration counts \#iter for the NS-FEM.
As expected, the condition number and iteration count increase when $n$ and $N$ increase, keeping the ratio $n/N$ constant for all the preconditioners considered.
Additionally, it is numerically confirmed that the preconditioners $\bar{M}^{-1}$ and $\bar{M}_{\alt}^{-1}$ reduce the condition number to some extent, whereas the preconditioner $M^{-1}$ does not.
\end{example}

\section{Conclusion}
\label{Sec:Conclusion}
Based on the fact that the stiffness matrices of the standard FEM, ES-FEM, and SSE are spectrally equivalent, we proved that any existing preconditioner for the standard FEM can be applied to the ES-FEM and SSE, inheriting good convergence properties such as numerical scalability.
We proposed the improved two-level additive Schwarz preconditioners for the ES-FEM and SSE.
Theoretically and numerically, the proposed preconditioners outperformed the standard one when they were applied to the ES-FEM and SSE.

This study suggests several interesting topics for future research.
The motivation for developing iterative solvers may influence their application to large-scale problems; we must solve more complex engineering problems on a large scale using FEMs with strain smoothing, equipped with the proposed preconditioners.
It is interesting to consider large-scale problems with oscillatory and high contrast coefficients~\cite{KRR:2015,KCW:2017}, which appear in the mathematical modeling of the flow in heterogeneous porous media.
Meanwhile, we observed in Section~\ref{Sec:NS} that the spectral property of the NS-FEM is different from the ES-FEM and SSE.
This proves that mathematical properties of the NS-FEM are somewhat different from those of other FEMs with strain smoothing.
Hence, developing a mathematical theory on the NS-FEM should be considered as a separate study.

\appendix
\section{Convergence theory of additive Schwarz methods}
\label{App:ASM}
In this appendix, we provide a brief summary on the abstract convergence theory of additive Schwarz methods introduced in~\cite{TW:2005,Park:2020}.
Let $V$ be a Hilbert space.
We consider the model linear problem
\begin{equation*}
Au = f,
\end{equation*}
where $A \colon V \rightarrow V$ is a symmetric and positive definite linear operator and $f \in V$.
In what follows, an index $j$ runs from $1$ to $\cN$.
For a Hilbert space $V_j$, we assume that there exists an interpolation operator $R_j^T \colon V_j \rightarrow V$ such that $V = \sum_{j=1}^{\cN} R_j^T V_j$.
Let $\tilde{A}_j \colon V_j \rightarrow V_j$ be a symmetric and positive definite linear operator which plays a role of a local operator on $V_j$.
In this setting, the additive Schwarz preconditioner $M^{-1} \colon V \rightarrow V$ is given by
\begin{equation*}
M^{-1} = \sum_{j=1}^{\cN} R_j^T \tilde{A}_j^{-1}R_j.
\end{equation*}
In order to obtain an upper bound for the condition number of the preconditioned operator $M^{-1}A$, we need the following three assumptions~\cite[Assumptions~4.9--4.11]{Park:2020}.

\begin{assumption}
\label{Ass:stable}
There exists a constant $C_0 > 0$ which satisfies the following: for any $v \in V$, there exists $v_j \in V_j$, $1 \leq j \leq \cN$, such that
\begin{equation*}
v = \sum_{j=1}^{\cN} R_j^T v_j
\end{equation*}
and
\begin{equation*}
\sum_{j=1}^{\cN} v_j^T \tilde{A}_j v_j \leq C_0^2 v^T A v.
\end{equation*}
\end{assumption}

\begin{assumption}
\label{Ass:CS}
There exists a constant $\tau_0 > 0$ which satisfies the following: for any $v_j \in V_j$, $1 \leq j \leq \cN$, and $\tau \in (0, \tau_0]$, we have
\begin{equation*}
\left( \sum_{j=1}^{\cN} R_j^T v_j \right)^T A \left( \sum_{j=1}^{\cN} R_j^T v_j \right)
\leq \frac{1}{\tau} \sum_{j=1}^{\cN} (R_j^T v_j)^T A (R_j v_j).
\end{equation*}
\end{assumption}

\begin{assumption}
\label{Ass:local}
There exists a constant $\omega_0 > 0$ which satisfies the following: for any $v_j \in V_j$, $1 \leq j \leq \cN$, we have
\begin{equation*}
(R_j^T v_j)^T A (R_j^T v_j) \leq \omega_0 v_j^T \tilde{A}_j v_j.
\end{equation*}
\end{assumption}

For detailed explanations of the assumptions above, we refer to~\cite{TW:2005,Park:2020}.
Under Assumptions~\ref{Ass:stable},~\ref{Ass:CS}, and~\ref{Ass:local}, we define the
additive Schwarz condition number $\kappa_{\ASM}$ as follows:
\begin{equation}
\label{kASM}
\kappa_{\ASM} = \frac{\omega_0 C_0^2}{\tau_0},
\end{equation}
where $C_0$, $\tau_0$, and $\omega_0$ are chosen as optimal as possible.
That is, $C_0$ is chosen as the minimum one satisfying Assumption~\ref{Ass:stable}, $\tau_0$ as the maximum one satisfying Assumption~\ref{Ass:CS}, and $\omega_0$ as the minimum one satisfying Assumption~\ref{Ass:local}.
The following theorem suggests that the convergence rate of a preconditioned iterative algorithm for $M^{-1}A$ relies on the additive Schwarz condition number $\kappa_{\ASM}$~\cite{Park:2020}.

\begin{theorem}
\label{Thm:ASM}
Under Assumptions~\ref{Ass:stable},~\ref{Ass:CS}, and~\ref{Ass:local}, we have
\begin{equation*}
\frac{\tau_0}{\omega_0} v^T A v \leq v^T M v \leq C_0^2 v^T A v \quad \forall v \in V.
\end{equation*}
Consequently, the following holds:
\begin{equation*}
\kappa (M^{-1}A) \leq \kappa_{\ASM},
\end{equation*}
where $\kappa_{\ASM}$ was given in~\eqref{kASM}.
\end{theorem}

\bibliographystyle{elsarticle-num-names} 
\bibliography{refs_SSE}

\end{document}